\documentclass[12pt, letterpaper]{amsart}

\newcommand{\indentalign}{\hspace{0.3in}&\hspace{-0.3in}}

\renewcommand{\Re}{\operatorname{Re}}
\renewcommand{\Im}{\operatorname{Im}}

\newcommand{\defeq}{\stackrel{\rm{def}}{=}}

\newcommand{\sgn}{\operatorname{sgn}}

\newcommand{\bear}{\begin{eqnarray}}
\newcommand{\eear}{\end{eqnarray}}
\newcommand{\beeq}{\begin{equation}}
\newcommand{\eneq}{\end{equation}}

\def\bm{\begin{bmatrix}}
\def\endm{\end{bmatrix}}

\def\de{\delta}

\let\phi\varphi

\usepackage{hyperref}
\hypersetup{
colorlinks=true,
linkcolor=black,
citecolor=black,
urlcolor=blue,
linktoc=page,
pdfstartview=FitH
}

\allowdisplaybreaks

\usepackage{amssymb}
\usepackage{amsthm}
\usepackage{amsxtra}
\usepackage{graphicx}

\newtheorem{theorem}{Theorem}

\newtheorem{lemma}[theorem]{Lemma}

\theoremstyle{remark}

\numberwithin{equation}{section}

\numberwithin{theorem}{section}

\numberwithin{table}{section}

\numberwithin{figure}{section}

\ifx\pdfoutput\undefined
  \DeclareGraphicsExtensions{.pstex, .eps}
\else
  \ifx\pdfoutput\relax
    \DeclareGraphicsExtensions{.pstex, .eps}
  \else
    \ifnum\pdfoutput>0
      \DeclareGraphicsExtensions{.pdf}
    \else
      \DeclareGraphicsExtensions{.pstex, .eps}
    \fi
  \fi
\fi

\title[Blow-up for 1D NLS with point nonlinearity II]{Blow-up for the 1D nonlinear Schr\"odinger equation with point nonlinearity II: \\ Supercritical blow-up profiles}

\author{Justin Holmer}
\author{Chang Liu}
\address{Brown University}

\begin{document}

\maketitle

\begin{abstract}
We consider the 1D nonlinear Schr\"odinger equation (NLS) with focusing \emph{point nonlinearity},
\begin{equation}
\label{E:intro-1}
i\partial_t\psi + \partial_x^2\psi + \delta|\psi|^{p-1}\psi = 0,
\end{equation}
where $\de=\de(x)$ is the delta function supported at the origin.   In the $L^2$ supercritical setting $p>3$, we construct self-similar blow-up solutions belonging to the energy space $L_x^\infty \cap \dot H_x^1$.  This is reduced to finding outgoing solutions of a certain stationary profile equation.  All outgoing solutions to the profile equation are obtained by using parabolic Weber functions and solving the jump condition at $x=0$ imposed by the $\delta$ term in \eqref{E:intro-1}.   This jump condition is an algebraic condition involving gamma functions, and existence and uniqueness of solutions is obtained using the intermediate value theorem and formulae for the digamma function.  We also compute the form of these outgoing solutions in the slightly supercritical case $0<p-3 \ll 1$ using the log Binet formula for the gamma function, and contour deformation and stationary phase/Laplace method in the integral formulae for the parabolic Weber functions.
\end{abstract}

\section{Introduction}
\label{S:introduction}

We consider the 1D nonlinear Schr\"odinger equation (NLS) with $p$-power focusing \emph{point nonlinearity}, for $p>1$,
\begin{equation}
\label{E:pNLS}
i\partial_t \psi + \partial_x^2 \psi + \delta |\psi|^{p-1}\psi =0
\end{equation}
where $\de=\de(x)$ is the delta function supported at the origin, and $\psi(x,t)$ is a complex-valued wave function for $x\in \mathbb{R}$.  The equation \eqref{E:pNLS} can be interpreted as the free linear Schr\"odinger equation
$$i\partial_t \psi + \partial_x^2 \psi =0 \,, \quad \text{for} \quad x\neq 0$$
together with the jump conditions at $x=0$:\footnote{We define $\psi(0-,t) = \lim_{x\nearrow 0} \psi(x,t)$ and $\psi(0+,t) = \lim_{x\searrow 0} \psi(x,t)$.}  
\begin{equation}
\label{E:jump}
\begin{aligned}
&\psi(0,t) \defeq \psi(0-,t) = \psi(0+,t) \\
&\partial_x \psi(0+,t) - \partial_x\psi (0-,t) = - |\psi(0,t)|^{p-1}\psi(0,t)
\end{aligned}
\end{equation}
The equation \eqref{E:pNLS} satisfies the scaling property
\begin{equation}
\label{E:scaling}
\psi(x,t) \text{ solves }\eqref{E:pNLS} \implies \psi_\lambda(x,t) = \lambda^{1/(p-1)} \psi(\lambda x, \lambda^2 t) \text{ solves }\eqref{E:pNLS}
\end{equation}
The scale invariant Sobolev space $\dot H^{\sigma_c}$, meaning the value $\sigma=\sigma_c$ for which $\| \psi_\lambda \|_{\dot H^\sigma}$ is $\lambda$-indepdendent, is $\sigma_c = \frac12 - \frac{1}{p-1}$, and we say the equation is $\dot H^{\sigma_c}$ critical.  The case $p=3$ is $L^2$ critical, and we say that $p<3$ is $L^2$ subcritical and $p>3$ is $L^2$ supercritical.  

We define the mass and energy of a solution $\psi$ to be
$$M(\psi(t)) = \| \psi(x,t) \|_{L_x^2}^2 \,, \qquad E(\psi(t)) = \frac12 \|\psi_x(x,t)\|_{L_x^2}^2 - \frac{1}{p+1} |\psi(0,t)|^{p+1}$$
By direct calculation, they are conserved, meaning that $M(\psi(t))$ and $E(\psi(t))$ are independent of $t$ whenever they are defined.  It follows from the Gagliardo-Nirenberg inequality
\begin{equation}
\label{E:GN}
\| \psi\|_{L_x^\infty}^2 \leq \| \psi\|_{L_x^2} \|\psi_x\|_{L_x^2}
\end{equation}
that if $p<3$, then all $H^1$ solutions to \eqref{E:pNLS} are global.  $H^1$ blow-up solutions do exist for $p\geq 3$.  In this paper, we seek explicit blow-up solutions for $p>3$ called \emph{self-similar} (since they are, up to a phase modulation, a rescaling of a fixed spatial profile).   It turns out to achieve exact self-similar blow-up solutions as in \eqref{E:self-similar-ansatz}, we have to relax the requirement that our solutions belong to $H_x^1$, and instead merely require that they belong to $L_x^\infty \cap \dot H_x^1$, which we call the \emph{energy space}, since the two terms defining the energy are finite for functions belonging to this space.

\begin{theorem}[structure of $L^2$ supercritical self-similar blow-up solutions]
\label{T:ansatz}
The function
\begin{equation}
\label{E:self-similar-ansatz}
\psi(x,t) = \lambda(t)^{1/(p-1)} e^{i\tau(t)} \eta(\lambda(t) x)
\end{equation}
solves \eqref{E:pNLS}  with $\lim_{t\nearrow T_*} \lambda(t) = +\infty$ if and only if there exists $h>0$ and $\kappa\in \mathbb{R}$ such that 
\begin{equation}
\label{E:lambda}
\lambda(t) = \frac{1}{\sqrt{ 2h (T_*-t)}} \,, \qquad \tau(t) = \frac{\kappa}{2h} \ln \left( \frac{T_*}{T_*-t} \right) + \tau(0) 
\end{equation}
and $\eta(z)$ solves the stationary equation
\begin{equation}
\label{E:blow-up-profile1}
(\kappa + ih \sigma_c) \eta - ih \Lambda_z \eta - \eta_{zz} -\delta |\eta|^{p-1}\eta=0 \,, \quad \Lambda_z = \frac12 + z \partial_z
\end{equation}
\end{theorem}

This is proved in \S\ref{S:ansatz}.  To better understand \eqref{E:blow-up-profile1}, let us drop the relationship between $\sigma_c$ and $p$ (given by $\sigma_c=\frac12-\frac{1}{p-1}$) and just consider the equation for arbitrary $\kappa\in \mathbb{R}$, $h>0$, and $\sigma\in \mathbb{R}$.
\begin{equation}
\label{E:blow-up-profile2}
(\kappa + ih \sigma) \eta - ih \Lambda_z \eta - \eta_{zz} -\delta |\eta|^{p-1}\eta=0
\end{equation}
Now we turn to a study of when \eqref{E:blow-up-profile2} has a solution in the energy space.  Taking $\eta(z) = e^{-\frac14 i z^2 h} \varphi(z)$, then $\eta(z)$ solves \eqref{E:blow-up-profile2} if and only if $\varphi(z)$ solves
\begin{equation}
\label{E:blow-up-profile3}
(\kappa+ih\sigma)\varphi - \varphi_{zz} - \frac14 h^2z^2 \varphi - \delta |\varphi|^{p-1}\varphi=0
\end{equation}

As discussed in \S\ref{S:Weber}, two independent solutions of the eigenvalue problem for the inverted harmonic well (here $\lambda\in \mathbb{C}$ is the spectral parameter and has nothing to do with $\lambda(t)$ in \eqref{E:lambda})
$$-w_{zz} - \frac14 h^2z^2 w = h\lambda w$$
are given by $w(z,\lambda)$, $w^*(z, \lambda)$ defined by \eqref{E:w-wstar} with the asymptotic behavior (see \eqref{E:2.10}, \eqref{E:2.11})
\begin{equation}
\label{E:asymp}
\begin{aligned}
& w(z) \sim (h^{1/2}z)^{i\lambda-\frac12} e^{\frac14 ihz^2} e^{\pi \lambda/4} e^{i\pi/8} \qquad && \text{as }z\to +\infty \\
& w^*(z) \sim  (h^{1/2}z)^{-i\lambda-\frac12} e^{-\frac14 ihz^2} e^{\pi \lambda/4} e^{-i\pi/8} \qquad  && \text{as }z\to +\infty 
\end{aligned}
\end{equation}
Thus, the most general solution $\varphi$ to \eqref{E:blow-up-profile3} is given by
\begin{equation}
\label{E:1-133}
\varphi(z) = 
\begin{cases}
\alpha_+ w(z) + \alpha_+^*w^*(z)  & \text{for }z>0\\
\alpha_- w(-z) + \alpha_-^*w^*(-z) & \text{for }z<0
\end{cases}
\end{equation}
for $\lambda = -\kappa h^{-1} - i\sigma$, where the four complex constants $\alpha_+$, $\alpha_+^*$, $\alpha_-$, $\alpha_-^*$ must be chosen to satisfy \eqref{E:jump}.  We will not address this general problem since we seek a solution $\varphi$ for which $\eta(z) = e^{-\frac14ihz^2} \varphi(z) \in L^\infty_z\cap \dot H^1_z$, which further constrains the problem, as we now explain.  From \eqref{E:asymp},
$$ 
\begin{aligned}
&|\partial_z (e^{-\frac14ihz^2} w(z,\lambda))| = O(|z|^{-\Im \lambda - \frac32}) \\
& |\partial_z( e^{-\frac14ihz^2} w^*(z,\lambda))| = O( |z|^{+\Im \lambda + \frac12})
\end{aligned}
\qquad  \text{as }z\to +\infty$$
and hence
$$\|e^{-\frac14ihz^2} w(z,\lambda) \|_{\dot H^1_{z>0}} < \infty \quad \iff \quad \Im \lambda >-1$$
$$\|e^{-\frac14ihz^2} w^*(z,\lambda) \|_{\dot H^1_{z>0}} < \infty \quad \iff \quad \Im \lambda <-1$$
With $\lambda = -\kappa h^{-1}-i\sigma$, we have $\Im \lambda = -\sigma$, and since we are interested in the case when $\sigma=\sigma_c<\frac12$, we must select solutions with no $w^*(z,\lambda)$ component.  

 A solution $\varphi$ to \eqref{E:blow-up-profile3} is called outgoing if, in \eqref{E:1-133}, $\alpha_+^*=0$ and $\alpha_-^* =0$.  Thus
\begin{equation}
\label{E:1-135}
\varphi(z) = 
\begin{cases}
\alpha_+ w(z)  & \text{for }z>0\\
\alpha_- w(-z)  & \text{for }z<0
\end{cases}
\end{equation}
By the first condition of \eqref{E:jump}, we must have $\alpha_+ = \alpha_-$.  Hence
$$\varphi(z) = \alpha w(|z|) $$
for some $\alpha\in \mathbb{C}$.  The second condition in \eqref{E:jump} becomes
$$2w_z(0,\lambda)= -|w(0,\lambda)|^{p-1}w(0,\lambda) |\alpha|^{p-1}$$
A solution $\alpha$ exists if and only if
\begin{equation}
\label{E:solve-condition}
2h^{1/2}A(\lambda)\defeq - \frac{2w_z(0,\lambda)}{w(0,\lambda)} \text{ is real and positive}
\end{equation}
and in this case, $\alpha$ must satisfy
\begin{equation}
\label{E:alpha}
|\alpha| = \left( - \frac{2w_z(0,\lambda)}{w(0,\lambda)} \right)^{1/(p-1)} \frac{1}{|w(0,\lambda)|}
\end{equation}
 Since \eqref{E:blow-up-profile3} is invariant under multiplication by $e^{i\theta}$ for $\theta\in \mathbb{R}$, the condition \eqref{E:solve-condition} uniquely specifies the solution once a choice of phase is given.  The most convenient choice is to take the phase of $\alpha$ to be that of $1/w(0,\lambda)$, so \eqref{E:alpha} becomes
\begin{equation}
\label{E:alpha2}
\alpha = \left( - \frac{2w_z(0,\lambda)}{w(0,\lambda)} \right)^{1/(p-1)} \frac{1}{w(0,\lambda)}
\end{equation}

Before examining the question of for which $\lambda$ (i.e. which $\kappa, \sigma$) the condition \eqref{E:solve-condition} holds, we can derive some general constraints.  We have established that for $\sigma<1$, $\eta(z)$ is a finite energy solution of \eqref{E:blow-up-profile2}, if and only if $\varphi(z)=e^{\frac14iz^2h}\eta(z)$ is an outgoing solution of \eqref{E:blow-up-profile3}.  Hence for $\sigma<1$, finite energy solutions $\eta(z)$ of \eqref{E:blow-up-profile2} have asymptotic behavior
\begin{equation}
\label{E:eta-asymp}
\eta(z) \sim c_{0,\lambda,h} |z|^{i\lambda -\frac12} \,, \qquad \partial_z\eta(z) \sim  (\sgn z)c_{1,\lambda,h} |z|^{i\lambda - \frac32} \,, \qquad \text{as }|z|\to \infty
\end{equation}
for certain constants $c_{0,\lambda,h}$, $c_{1,\lambda,h}$, with $\lambda = -\kappa h^{-1} - i \sigma$.    

\begin{theorem}[Pohozhaev identities]
\label{T:Pohozhaev}
If $\sigma<1$ and $\eta(z)$ is a nontrivial finite energy solution of \eqref{E:blow-up-profile2}, then $\sigma> 0$ and the following identity holds:
$$0 = (1-\sigma) \int_{-\infty}^{+\infty} |\eta_z|^2 - \left( \frac12-\sigma\right) |\eta(0)|^{p+1}$$
from which we obtain that $E(\eta)=0$ when $\sigma=\sigma_c$.
\end{theorem}

This is proved in \S\ref{S:Pohozhaev}, and constrains the admissible value for $\sigma$ to be $0<\sigma<1$ (for outgoing solutions).  For a particular choice of $h$, $\kappa$, and $\sigma$, let us denote by $\varphi_{h,\kappa,\sigma}$ the unique nontrivial outgoing solution to \eqref{E:blow-up-profile3} (if it exists).   By scaling we find the relation
$$\varphi_{h,\kappa,\sigma}(z) = \mu^{1/(p-1)} \varphi_{\mu^{-2}h, \mu^{-2}\kappa, \sigma}(\mu z) \,, \qquad \mu>0$$
Hence, by taking $\mu = |\kappa|^{1/2}$, we can convert $\kappa$ to $\pm 1$ while $h$ is converted to $|\kappa|^{-1}h$.  
At this point, we appeal to the special function representation of $w(z,\lambda)$ to compute a formula for $A(\lambda)$ defined in \eqref{E:solve-condition}.  We find (see \eqref{E:A})
$$A(\lambda) = \frac{e^{-i\pi/4} \sqrt{2} \Gamma(\frac34-\frac12i\lambda)}{\Gamma(\frac14-\frac12i\lambda)}$$

\begin{theorem}[existence and uniqueness of outgoing solutions]
\label{T:profiles-kappa+1}
Recall $\lambda = -\kappa h^{-1}-i\sigma$. 
If $\kappa=1$, then for each $0<h<\infty$, there exists a unique $0<\sigma(h)<1$ such that $A(\lambda)$ is real and positive, and thus a corresponding outgoing solution $\varphi_{h,1,\sigma(h)}$.    On the other hand, if $\kappa=-1$, then for each $h>0$ and $0<\sigma<1$, $A(\lambda)$ is \emph{not} real and positive, and thus there are no outgoing solutions $\varphi_{h,-1,\sigma}$.   
\end{theorem}

Theorem \ref{T:profiles-kappa+1} is proved in \S\ref{S:profiles}.   For $\kappa=1$, we numerically solved for $\sigma(h)$ using the MATLAB \texttt{fzero} function.  The results are displayed in Figure \ref{F:1}, and show that the $h\to 0$ asymptotic formula $\sigma(h) = 2e^{-\pi/h} h^{-1} (1+O(h))$ given in the next theorem is already a good approximation at $h=1$.  Moreover, the numerical solution shows that $\sigma(h^{-1})$ is decreasing as a function of $h^{-1}$, starting from $\sigma=\frac12$ at the limit $h^{-1}=0$.  Thus we have the numerical finding that $0<\sigma(h)<\frac12$ for all $h>0$ (as opposed to just $0<\sigma(h)<1$) and moreover, that for each $0<\sigma<\frac12$, there exists a unique $0<h<\infty$ such that $\sigma = \sigma(h)$.  

\begin{theorem}[asymptotic formulae for $\sigma(h)$]
\label{T:profiles-asymptotic}
Let $\kappa=1$.  For all $0<h \ll 1$,  the unique value $\sigma(h)$ such that $A(\lambda)$ is real and positive (as in Theorem \ref{T:profiles-kappa+1}) is
$$\sigma(h) = 2e^{-\pi/h} h^{-1} (1+O(h))$$
For $h \gg 1$ (i.e. $0<h^{-1}\ll 1$), 
$$\sigma(h) = \tfrac12 - h^{-1} + O(h^{-2})$$
\end{theorem}

This is proved in \S\ref{S:profiles-asymptotic} and numerically confirmed in Fig. \ref{F:1}.

\begin{figure}
\includegraphics[scale=0.4]{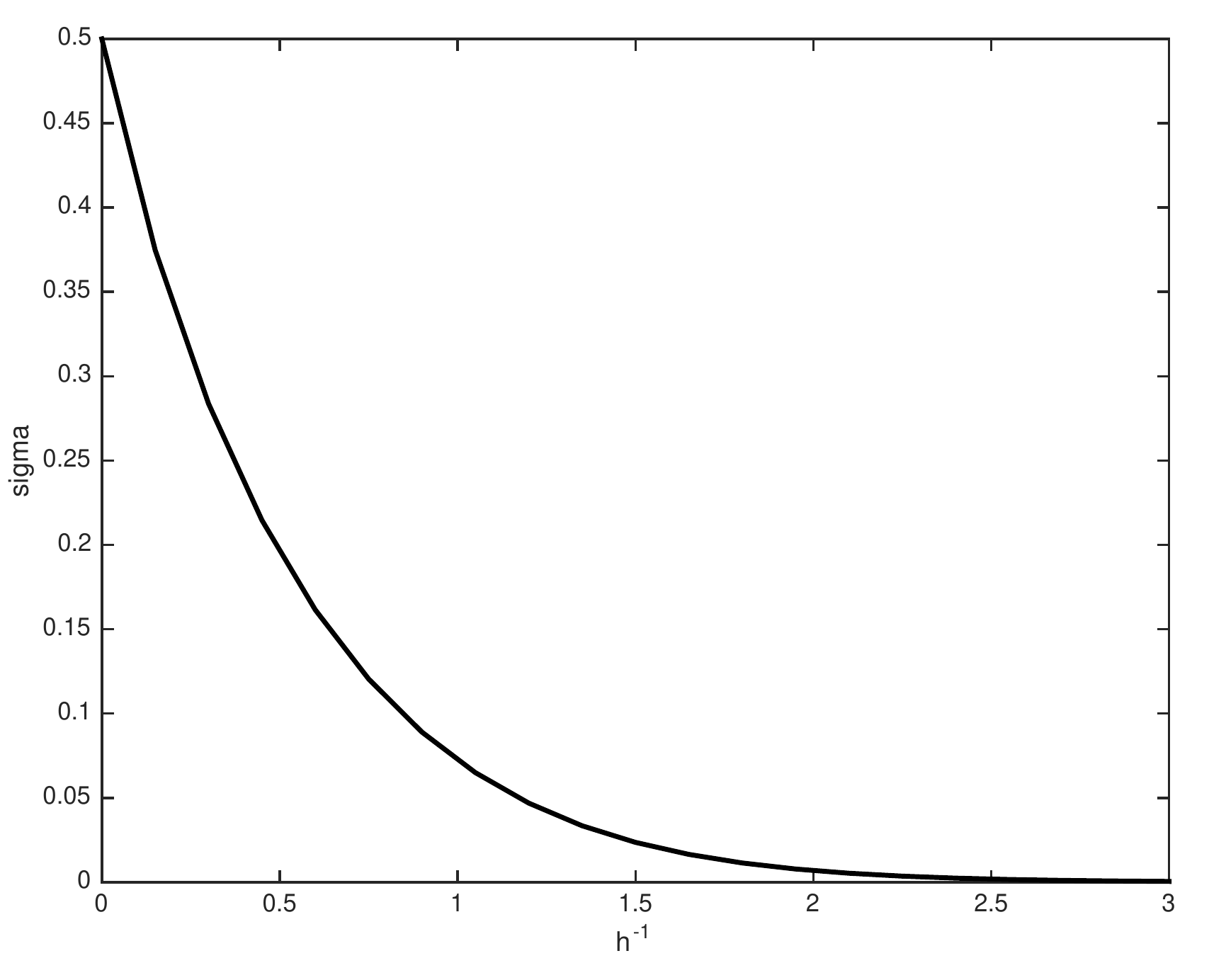}
\includegraphics[scale=0.4]{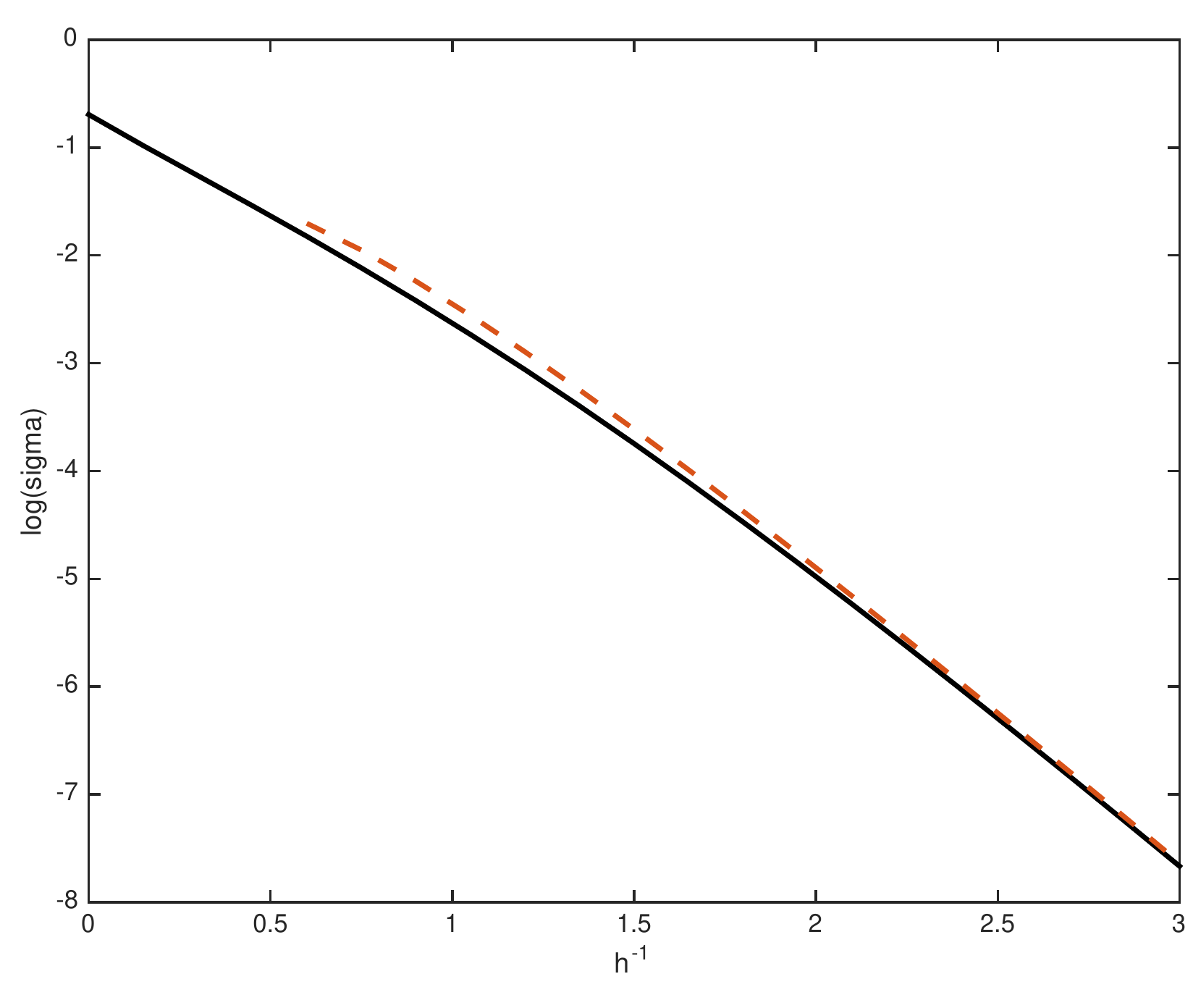}
\caption{ \label{F:1} (left) plot of numerical solution of $\sigma(h)$ versus $h^{-1}$, (right) plot of numerical solution of $\log \sigma(h)$ versus $h^{-1}$ in solid black, together with $h^{-1}\to \infty$ asymptotic $\log \sigma(h) = \log 2 -\pi h^{-1} + \log h^{-1}$, from Theorem \ref{T:profiles-asymptotic}, in dashed red, showing good agreement for $h^{-1}>1$.}
\end{figure}

Recalling that outgoing solutions $\varphi$ to \eqref{E:blow-up-profile3} correspond to finite energy (hence zero energy) solutions $\eta$ to \eqref{E:blow-up-profile2}, we still need to invert the relationship between $h$ and $\sigma$.  Indeed, in Theorem \ref{T:ansatz}, we are given $p>3$, and hence $0<\sigma_c<\frac12$, and need to find $h$ such that $\sigma(h)=\sigma_c$.  For $p\approx 3$, the first-order approximation is $ h \approx \pi [ \log 8/(p-3) ]^{-1}$.

In the next theorem, for given $0<h \ll 1$ and $\sigma(h)=2e^{-\pi/h} h^{-1} (1+O(h))$, we employ contour deformation and stationary phase in the parabolic Weber functions to better understand the shape of the outgoing solution $\varphi(z) = \varphi_{h,1,\sigma(h)}(z)$.

\begin{theorem}[outgoing profile $\varphi$ asymptotics as $h\to 0$]
\label{T:WKB}
As $h\to 0$, we have the expansion
$$ \varphi(x) \approx 
\begin{cases}
2^{1/2} (1-\tfrac14h^2x^2)^{-1/4} e^{-h^{-1}[\arcsin(\frac12 hx) + \frac12 hx(1-\frac14h^2x^2)^{1/2}]} & \text{for }x\ll 2h^{-1} \\
c_h e^{\frac14 ix^2} x^{-ih^{-1}+\sigma-\frac12} & \text{for }x\gg 2h^{-1}
\end{cases}
$$
where $c_h = 2 e^{\pi i/4} e^{-\frac12 i h^{-1}} e^{-\pi/(2h)}$.
\end{theorem}

This is proved in \S\ref{S:WKB}.

\subsection*{Acknowledgments}

We thank Catherine Sulem, Galina Perelman, and Maciej Zworski for discussions about this topic, suggestions, and encouragement.  The material in this paper will be included as part of the PhD thesis of the second author at Brown University.  While this work was completed, the first author was supported in part by NSF grants DMS-1200455,  DMS-1500106.  The second author was supported in part by NSF grant DMS-1200455 (PI Justin Holmer).

\section{$L^2$ supercritical blow-up ansatz}
\label{S:ansatz}

In this section we prove Theorem \ref{T:ansatz}.   Using the scaling property \eqref{E:scaling} as a model, we examine \emph{self-similar} solutions of the form
$$\psi(x,t) = \lambda(t)^{1/(p-1)} u( \lambda(t) x, \tau(t))$$
We convert the equation \eqref{E:pNLS} into an equation for $u(z,\tau)$ by computing
\begin{align*}
\indentalign i\psi_t + \psi_{xx} + \delta |\psi|^{p-1}\psi  \\
&= \lambda^{\frac{1}{p-1}+2} \left( i \lambda^{-2} \frac{d\tau}{dt} \partial_\tau u + i \lambda^{-3} \lambda_t (-\sigma_c + \Lambda) u + u_{zz} + \delta |u|^{p-1}u \right)
\end{align*}
where $\Lambda \defeq \frac12 + z\partial_z$.  If $u(z,\tau)= e^{i\tau}\eta(z)$, then \eqref{E:pNLS} holds if and only if 
\begin{equation}
\label{E:ansatz1}
0 = - \lambda^{-2} \frac{d\tau}{dt} \eta + i \lambda^{-3}\lambda_t(-\sigma_c+\Lambda) \eta + \eta_{zz} + \delta |\eta|^{p-1}\eta
\end{equation}
Let $\kappa(t) = \lambda^{-2} \frac{d\tau}{dt}$ and $h(t) = \lambda^{-3}\lambda_t$ so that
\begin{equation}
\label{E:ansatz2}
0 = - \kappa(t) \eta + i h(t)(-\sigma_c+\Lambda) \eta + \eta_{zz} + \delta |\eta|^{p-1}\eta
\end{equation}
Then $\kappa(t)$ and $h(t)$ are in fact constant.  Indeed, for $t_1\neq t_2$, we subtract \eqref{E:ansatz2} at the two times to obtain
$$  i(h(t_1)-h(t_2))(-\sigma_c+\Lambda)\eta = (\kappa(t_1)-\kappa(t_2)) \eta$$
which can rewritten as
$$i(h(t_1)-h(t_2))  z f'(z)= (\kappa(t_1)-\kappa(t_2)) f(z) \,, \qquad f(z)=z^{1/(p-1)}\eta(z)$$
for which there are no nontrivial solutions compatible with \eqref{E:ansatz2} at $t=t_1$.  This forces $\kappa(t_1)=\kappa(t_2)$ and $h(t_1)=h(t_2)$, and hence $\kappa(t)=\kappa$ and $h(t)=h$ are constant and \eqref{E:ansatz2} becomes \eqref{E:blow-up-profile1}.

Since $-\frac12(\lambda^{-2})_t = h$ and $\lambda^{-1}(T_*)=0$, we integrate to obtain $\lambda^{-2}(t) = 2h(T_*-t)$.  Since we want to approach the blow-up time $T_*$ from below ($t<T_*$), we have that $h>0$ and the first equation in \eqref{E:lambda} holds.  Combining this with $\kappa = \lambda^{-2} \frac{d\tau}{dt}$ gives 
$$\frac{d\tau}{dt} = \lambda^2 \kappa = \frac{\kappa}{2h(T_*-t)} \,,$$
 and integrating gives the second equation in \eqref{E:lambda}.  

\section{Parabolic cylinder functions}
\label{S:Weber}

The following material is drawn in part from Slavyanov \cite{MR1398655}, pp. 21-31.  Consider the \emph{Weber equation}\footnote{This equation appears as (1) in \S 8.1 of p. 116 of the Bateman Manuscript Project, Higher Transcendental Functions, Volume II.   It is also written in \cite{MR1398655} as (3.1) on p. 22.}
\begin{equation}
\label{E:2.4}
\gamma'' + (\nu + \tfrac12 - \tfrac14 z^2) \gamma=0
\end{equation}
Any solution to \eqref{E:2.4} is called a \emph{parabolic cylinder function} or \emph{Weber-Hermite function}.  One solution to \eqref{E:2.4} is $\gamma(z)=D_\nu(z)$, defined for $\Re \nu <0$ by the integral formula\footnote{This appears as item (3) in \S 8.3 of p. 119 of the Bateman Manuscript Project, Higher Transcendental Functions, Volume II.  It also appears as (3.2) on p. 22 of \cite{MR1398655}.}
\begin{equation}
\label{E:2.3}
D_\nu(z) = \frac{e^{-\frac14 z^2}}{\Gamma(-\nu)} \int_0^\infty e^{-zt} e^{-\frac12 t^2} t^{-\nu-1} \, dt
\end{equation}
  The function $D_\nu(z)$ can be extended analytically to all $\nu \in \mathbb{C}$, but for now we do not need formulae for $\Re \nu \geq 0$.  The fact that \eqref{E:2.3} is a solution to \eqref{E:2.4} can be verified by direct computation of the second derivative of \eqref{E:2.3} (differentiation under the integral sign and integration by parts).  We calculate
\begin{equation}
\label{E:Dzero}
D_\nu(0) = \frac{1}{\Gamma(-\nu)} \int_0^\infty e^{-\frac12t^2} t^{-\nu-1} \, dt = \frac{2^{-\frac{\nu}{2}-1}\Gamma(-\frac{\nu}{2})}{\Gamma(-\nu)} =  \frac{ \sqrt{\pi} \, 2^{\frac{\nu}{2}}}{\Gamma(\frac{1-\nu}{2})}
\end{equation}
\begin{equation}
\label{E:Dzeroprime}
D_\nu'(0) = - \frac{1}{\Gamma(-\nu)} \int_0^\infty e^{-\frac12 t^2} t^{-\nu} \, dt 
= - \frac{ 2^{-\frac{\nu}{2}-\frac12} \Gamma(-\frac{\nu}{2}+\frac12) }{\Gamma(-\nu)}
= - \frac{\sqrt{\pi} \, 2^{\frac{\nu}{2}+\frac12}}{\Gamma(-\frac{\nu}{2})}
\end{equation}
For both formulae we have used the duplication formula
$$\Gamma(\nu) = \frac{2^{\nu-1}}{\sqrt{\pi}} \Gamma(\tfrac{\nu}{2})\Gamma(\tfrac{\nu}{2}+\tfrac12)$$
The formula \eqref{E:2.3} applies for $\Re \nu <0$, but for one remark, we do need to know the solution $D_\nu(z)$ to \eqref{E:2.4} at $\nu=0$.   From the formulae \eqref{E:Dzero} and \eqref{E:Dzeroprime}, we obtain $D_0(0)=1$ and $D_0'(0)=0$.  It is straightforward by direct computation to confirm that $\gamma(z) = e^{-z^2/4}$ is the unique solution to \eqref{E:2.4} satisfying these initial conditions, so we conclude that $D_0(z) = e^{-z^2/4}$.

It is straightforward to verify that since $\gamma_1(z) =D_\nu(z)$ solves \eqref{E:2.4}, so do the two functions $\gamma_2(z)= D_\nu(-z)$ and $\gamma_3(z)=D_{-\nu-1}(iz)$.  Let
$$
C=  
\begin{bmatrix}
\gamma_2(0) & \gamma_3(0) \\
\gamma_2'(0) & \gamma_3'(0)
\end{bmatrix}
=
\sqrt{\pi}
\begin{bmatrix}
 \frac{2^{\nu/2}}{\Gamma( \frac{1-\nu}{2} )} & \frac{2^{(-\nu-1)/2}}{\Gamma(\frac{2+\nu}{2})} \\
 \frac{2^{(\nu+1)/2}}{\Gamma(-\frac{\nu}{2})} &
-i \frac{2^{-\nu/2}}{\Gamma(\frac{1+\nu}{2})}
\end{bmatrix}
$$
The Wronskian is given by
$$ W(\gamma_3, \gamma_2) = \det C = \frac{-i\pi}{\Gamma( \frac{1-\nu}{2}) \Gamma( \frac{1+\nu}{2})} - \frac{\pi}{\Gamma( \frac{2+\nu}{2}) \Gamma( - \frac{\nu}{2})}$$
To simplify this, we use the standard gamma function identity
\begin{equation}
\label{E:2.7}
\Gamma(z) \Gamma(1-z) = \frac{\pi}{\sin \pi z}
\end{equation}
Taking $z = \frac{1+\nu}{2}$ in \eqref{E:2.7}, we obtain
$$\Gamma( \frac{1-\nu}{2}) \Gamma( \frac{1+\nu}{2}) = \frac{\pi}{\sin \pi( \frac{1-\nu}{2})} = \frac{\pi}{\cos \frac{\pi \nu}{2}}$$
Taking $z= \frac{2+\nu}{2}$ in \eqref{E:2.7}, we have
$$\Gamma( \frac{2+\nu}{2}) \Gamma( - \frac{\nu}{2} ) = \frac{\pi}{\sin \pi ( \frac{2+\nu}{2})}= - \frac{\pi}{\sin \frac{\pi\nu}{2}}$$
Substituting, we obtain\footnote{This Wronskian formula agrees with the $+$ case of (3.13) of \cite{MR1398655} on p. 24.}
\begin{equation}
\label{E:Wronsk}
W(\gamma_3,\gamma_2) = -i\cos\frac{\pi \nu}{2} + \sin \frac{\pi \nu}{2}= -i e^{i\pi \nu/2}
\end{equation}
Thus, for all $\nu \in \mathbb{C}$, $W(\gamma_2,\gamma_3)\neq 0$ and $\{ \, \gamma_2, \gamma_3 \, \}$ is a basis for the space of solutions to \eqref{E:2.4}.  It follows that there exist $\alpha$, $\beta$ such that 
$$\gamma_1(z) = \alpha\gamma_2(z) + \beta\gamma_3(z)$$
The constants $\alpha$, $\beta$ are found by taking $z=0$ to obtain the system
$$\begin{bmatrix} \alpha \\ \beta \end{bmatrix} = C^{-1} \begin{bmatrix} \gamma_1(0) \\ \gamma_1'(0) \end{bmatrix} 
= i\pi e^{-i\pi \nu/2} \begin{bmatrix} \frac{-i}{\Gamma(\frac{1+\nu}{2}) \Gamma( \frac{1-\nu}{2})} + \frac{1}{\Gamma(\frac{2+\nu}{2})\Gamma(-\frac{\nu}{2})} \\
-\frac{ 2^{\nu+\frac32}}{\Gamma(-\frac{\nu}{2}) \Gamma( \frac{1-\nu}{2})} \end{bmatrix} 
$$
Using the Gamma function identities as before, we simplify to
$$\begin{bmatrix} \alpha \\ \beta \end{bmatrix} = ie^{-i\pi \nu/2} \begin{bmatrix}  -i e^{-i\pi \nu/2} \\ -\frac{\sqrt{2\pi}}{ \Gamma(-\nu)} \end{bmatrix}  $$
Hence we have\footnote{This agrees with Slavyanov \cite{MR1398655} equation (3.14) on p. 24}
\begin{equation}
\label{E:2.6}
D_\nu(z) = e^{-i\pi \nu} D_\nu(-z) + e^{-i\pi (\nu+1)/2} \frac{\sqrt{2\pi}}{\Gamma(-\nu)}D_{-\nu-1}(iz)
\end{equation}
For real $\nu$ and $z$, we can take the complex conjugate to obtain
\begin{equation}
\label{E:2.6b}
D_\nu(z) = e^{i\pi \nu} D_\nu(-z) + e^{i\pi (\nu+1)/2} \frac{\sqrt{2\pi}}{\Gamma(-\nu)}D_{-\nu-1}(-iz)
\end{equation}
However, \eqref{E:2.6b} remains valid for all $\nu, z \in \mathbb{C}$ by analytic continuation.

Consider now the scalar Schr\"odinger operator with inverted harmonic potential.  Two solutions to the second-order ODE 
\begin{equation}
\label{E:2.1}
-\partial_x^2 v-\tfrac14 x^2 v =\lambda v
\end{equation}
are given by
\begin{equation}
\label{E:2.2}
v(x) = D_{i\lambda -\frac12} ( e^{-i\pi/4} x)
\end{equation}
\begin{equation}
\label{E:2.8}
v^*(x) = D_{-i\lambda - \frac12}(e^{i\pi/4}x)
\end{equation}
The fact that $v(x)$ and $v^*(x)$ solve \eqref{E:2.1} is verified by using the fact that $D_\nu(z)$ is a solution to \eqref{E:2.4}.  

In this section, we record some properties of $v(x)$ and $v^*(x)$.    In particular, we obtain the asymptotics of $v(x)$ and $v^*(x)$ as $x\to \pm \infty$, the values of $v(0)$, $v^*(0)$, $v'(0)$, $(v^*)'(0)$, and the Wronskians $W[v(x),v(-x)]$ and $W[v(x),v^*(x)]$.  

Substituting \eqref{E:2.3} into \eqref{E:2.2} with $\nu = i\lambda - \frac12$, we obtain the integral formula
\begin{equation}
\label{E:2.5}
v(x) = \frac{e^{\frac14 ix^2}}{\Gamma(-i\lambda + \frac12)} \int_0^\infty e^{-\frac12 t^2} e^{- \big( \frac{1-i}{\sqrt{2}}\big) tx} \, t^{-i\lambda- \frac12} \, dt \,, \qquad \Im \lambda>-\frac12
\end{equation}
Substituting \eqref{E:2.3} into \eqref{E:2.8} with $\nu = -i\lambda - \frac12$, we obtain the integral formula
\begin{equation}
\label{E:2.5b}
v^*(x) = \frac{e^{-\frac14 ix^2}}{\Gamma(i\lambda + \frac12)} \int_0^\infty e^{-\frac12 t^2} e^{- \big( \frac{1+i}{\sqrt{2}}\big) tx} \, t^{i\lambda- \frac12} \, dt\,, \qquad \Im \lambda<\frac12
\end{equation}

Now we evaluate the asymptotics of $v(x)$ given by \eqref{E:2.5} as $x\to +\infty$.  Taking $s= e^{-i\pi/4} x t$ (which amount to a contour change that is valid for $x>0$), we obtain
\begin{equation}
\label{E:2.10-}
v(x) = \frac{ x^{i\lambda-\frac12}  e^{\frac14 ix^2} e^{\pi \lambda/4} e^{i\pi/8} }{ \Gamma( -i\lambda+\frac12)} v_+(x)\end{equation}
where
\begin{align*}
v_+(x) &= \int_0^\infty e^{-\frac{is^2}{2x^2}} e^{-s} s^{-i\lambda-\frac12} \, ds \\
&= \int_0^\infty e^{-s}e^{-i\lambda-\frac12}\, ds + \int_0^{\infty} (e^{-\frac{is^2}{2x^2}}-1) e^{-s} s^{-i\lambda-\frac12} \, ds\\
&= \Gamma(-i\lambda+\tfrac12) + v_{+,1}(x) 
\end{align*}
Moreover, we have the crude estimate
$$|v_{+,1}(x)| \leq \int_0^{+\infty} \frac{s^2}{2x^2} e^{-s} s^{\Im \lambda - \frac12} \, ds=\tfrac12 x^{-2} \Gamma(\Im \lambda +\tfrac52)$$
Thus we have\footnote{This agrees with the result in the table on p. 29, top entry, in \cite{MR1398655}.}
\begin{equation}
\label{E:2.10}
v(x) \sim x^{i\lambda-\frac12} e^{\frac14 ix^2} e^{\pi \lambda/4} e^{i\pi/8}\left(1+ \frac{ \Gamma(\Im \lambda + \frac52)}{\Gamma(-i\lambda+\frac12)}O(x^{-2})\right) \quad \text{as }x\to +\infty
\end{equation}
uniformly in $\lambda$.    In the case where $\lambda = -h^{-1}-i\sigma$ with $0<h\ll 1$ and $0<\sigma<\frac12$, we have $|\Gamma(-i\lambda+\frac12)| \approx \pi h^\sigma e^{-\pi/(2h)}$ and thus $\Gamma(\Im \lambda+\frac52)/\Gamma(-i\lambda+\frac12) \sim h^{-\sigma} e^{\pi/2h}$, from which it follows that we need $x \gg h^{2\sigma} e^{\pi/(4h)}$ in order for \eqref{E:2.10} to apply.   A more precise result is given in Lemma \ref{L:asymp} below, showing that the leading order term in \eqref{E:2.10} is valid even for $x \gg 2h^{-1/2}$.  

A similar calculation applied to $v^*(x)$ given by \eqref{E:2.5b}  yields
\begin{equation}
\label{E:2.11}
v^*(x) \sim  x^{-i\lambda-\frac12} e^{-\frac14 ix^2} e^{\pi \lambda/4} e^{-i\pi/8}\left(1+\frac{\Gamma(-\Im \lambda+\frac52)}{\Gamma(i\lambda+\frac12)} O(x^{-2})\right) \quad \text{as }x\to +\infty 
\end{equation}
uniformly in $\lambda$.  

Now we evaluate the asymptotics of $v(x)$ given by \eqref{E:2.5} as $x\to -\infty$.   By \eqref{E:2.6b},
\begin{equation}
\label{E:2.9}
v(x) = -ie^{-\pi \lambda} v(-x) + \frac{\sqrt{2\pi}}{\Gamma(-i\lambda+\frac12)} e^{ -\frac{\pi \lambda}{2}} e^{i\pi/4} v^*(-x)
\end{equation}
By \eqref{E:2.6} with $\nu = -i\lambda-\tfrac12$, $-\nu-1=i\lambda-\tfrac12$, $z=e^{i\pi/4}x$, $iz=-e^{-i\pi/4}x$, we obtain
\begin{equation}
\label{E:2.9b}
v^*(x) = ie^{-\pi\lambda} v^*(-x) + e^{-\frac14 i \pi} e^{-\frac12 \pi \lambda} \frac{\sqrt{2\pi}}{\Gamma(i\lambda+\tfrac12)} v(-x)
\end{equation}

We can use \eqref{E:2.9}, \eqref{E:2.10}, \eqref{E:2.11} to obtain\footnote{This agrees with the second row of the table on p. 29 of \cite{MR1398655}.  Here we take $z= e^{3\pi i/4}x$}
\begin{equation}
\label{E:v-left-asymp}
\begin{aligned}
v(x) \sim  &-i e^{\frac{\pi}{8}i} e^{-\frac{3\pi \lambda}{4}} (-x)^{i\lambda-\frac12} e^{\frac14ix^2} \\
& \quad + \frac{\sqrt{2\pi}}{\Gamma(-i\lambda+\frac12)} e^{\frac{\pi}{8}i} e^{-\frac{\pi \lambda}{4}} (-x)^{-i\lambda - \frac12} e^{-\frac14 ix^2} 
\end{aligned}
\quad \text{as }x\to -\infty
\end{equation}
Although these asymptotics were derived under the constraint that $-\frac12 < \Im \lambda < \frac12$, they continue to all $\lambda \in \mathbb{C}$ by analytic continuation.


We also note that
\begin{equation}
\label{E:2.49}
v(0) = D_{i\lambda-\frac12}(0)= \frac{ \sqrt{\pi} 2^{\frac12 i \lambda - \frac14}}{\Gamma( \frac34 - \frac12 i \lambda)}
\end{equation}
\begin{equation}
\label{E:2.50}
v'(0) = e^{-i\pi/4}D_{i\lambda-\frac12}'(0) = - e^{-i\pi/4} \frac{\sqrt{\pi} 2^{\frac12 i \lambda + \frac14}}{\Gamma(-\frac12 i \lambda + \frac14)}
\end{equation}
\begin{equation}
\label{E:2.53}
v^*(0) = D_{-i\lambda-\frac12}(0)= \frac{ \sqrt{\pi} 2^{-\frac12 i \lambda - \frac14}}{\Gamma( \frac34 + \frac12 i \lambda)}
\end{equation}
\begin{equation}
\label{E:2.54}
(v^*)'(0) = e^{i\pi/4}D_{-i\lambda-\frac12}'(0) = - e^{i\pi/4}\frac{\sqrt{\pi} 2^{-\frac12 i \lambda + \frac14}}{\Gamma(\frac12 i \lambda + \frac14)}
\end{equation}

We conclude with some Wronskians. First
$$W[v(x),v(-x)] = 2v'(0)v(0) = - \frac{ e^{-i\pi/4} \pi \, 2^{i\lambda+1}}{\Gamma(\frac14-\frac12i\lambda)\Gamma(\frac34-\frac12i\lambda)}$$
Using the identity $\sqrt \pi 2^{1-2z} \Gamma(2z) = \Gamma(z)\Gamma(z+\frac12)$ with $z=\frac14-\frac12i\lambda$, this simplifies to
$$W[v(x),v(-x)] = - \frac{ e^{-i\pi/4} \sqrt{2\pi}}{\Gamma(\frac12-i\lambda)}$$

Now consider
\begin{align*}
W[v(x),v^*(x)] &= e^{-i\pi/4} W[ D_\nu(z), D_{-\nu-1}(iz)]  \\
&= e^{-i\pi/4} W[D_{-\nu-1}(-iz),D_\nu(-z)] = e^{-i \pi/4} W[\gamma_3,\gamma_2]
\end{align*}
where $z= e^{-i\pi/4}x$ and $\nu = i\lambda-\frac12$.  Substituting \eqref{E:Wronsk}, we obtain
\begin{equation}
\label{E:2.47}
W[v(x),v^*(x)] = i e^{\pi\lambda/2}
\end{equation}
  
Now that we have laid out the basic properties of the fundamental solutions $v(x,\lambda)$ and $v^*(x,\lambda)$ of \eqref{E:2.1}, we consider the scaled Hamiltonian, and associated eigenvalue problem
$$-\partial_x^2 w - \tfrac14 h^2 x^2 w = h \lambda w$$
Two solutions are given by
\begin{equation}
\label{E:w-wstar}
w(x,\lambda) = v(h^{1/2} x, \lambda) \,, \qquad w^*(x, \lambda) = v^*(h^{1/2}x,\lambda)
\end{equation}
The basic properties of $w(x,\lambda)$ and $w^*(x,\lambda)$ are easily deduced from the corresponding properties for 
$v(x,\lambda)$ and $v^*(x,\lambda)$ given above.

\section{Outgoing solutions and the Pohozhaev identities}
\label{S:Pohozhaev}

In this section, we prove Theorem \ref{T:Pohozhaev}, that is, we derive the Pohozhaev identities for finite energy solutions $\eta(z)$ of \eqref{E:blow-up-profile1}, for $\sigma<1$, and deduce some consequences.  By the analysis given in \S \ref{S:introduction}, any such $\eta(z)$ has the $|z|\to \infty$ asymptotics given by \eqref{E:eta-asymp} with $\lambda = -\kappa h^{-1} - i\sigma$.  The profile equation \eqref{E:blow-up-profile2} is equivalent to 
\begin{equation}
\label{E:eta-reform}
(\kappa +ih\sigma) \eta -ih\Lambda \eta - \eta_{zz} =0 \qquad \text{for }z\neq 0
\end{equation}
together with the juncture conditions at $z=0$ given by
\begin{equation}
\label{E:eta-juncture}
\begin{aligned}
& \eta(0) = \eta(0-) = \eta(0+) \\
& \eta_z(0+)-\eta_z(0-) = -|\eta(0)|^{p-1}\eta(0)
\end{aligned}
\end{equation}
Also since $\eta(z) = \alpha w(|z|,\lambda)$, for some $\alpha$, where $w$ is smooth across $z=0$, we calculate
\begin{align*}
&\eta_z(z) = \alpha w'(|z|) \sgn z \\
&\eta_{zz}(z) = \alpha w''(|z|) + 2\alpha w'(0) \delta(z) \\
& \eta_{zzz}(z) = \alpha w'''(|z|)\sgn z + 2\alpha w'(0) \delta'(z) 
\end{align*}
which shows that the left-hand and right-hand limits for all derivatives exist and in particular
$$\eta_{zz}(0) \defeq \eta_{zz}(0-)=\eta_{zz}(0+)$$
Recall that we also derived the asymptotics \eqref{E:eta-asymp}.  
These properties are used to derive the Pohozhaev identities below.

  Pair \eqref{E:eta-reform} with $\bar \eta$, and integrate over $-R< z < R$ to obtain
\begin{equation}
\label{E:Poh1}
(\kappa+ih\sigma) \int_{-R}^R |\eta|^2 - i h \int_{-R}^R \Lambda \eta \, \bar \eta - \int_{-R}^R \eta_{zz} \bar \eta =0
\end{equation}
Take the real part of \eqref{E:Poh1}, using that $-\Re \eta_{zz} \bar \eta = -\frac12 \partial_z^2 |\eta|^2 + |\eta_z|^2$, applying the fundamental theorem of calculus on $(0,R)$ and $(-R,0)$,
$$
0 = 
\begin{aligned}[t]
&\kappa \int_{-R}^R |\eta|^2 + h \Im \int_{-R}^R \Lambda \eta \, \bar \eta  + \int_{-R}^R |\eta_z|^2 
 \\
 &- \frac12 \left[ \partial_z |\eta|^2 \right]_{-R}^R + \Re \overline{\eta(0)}(\eta'(0+)-\eta'(0-))
 \end{aligned}
 $$
Using \eqref{E:eta-juncture} across $z=0$,
$$
0 = 
\kappa \int_{-R}^R |\eta|^2 + h \Im \int_{-R}^R \Lambda \eta \, \bar \eta  + \int_{-R}^R |\eta_z|^2 
- \frac12 \left[ \partial_z |\eta|^2 \right]_{-R}^R -|\eta(0)|^{p+1}
  $$
 By \eqref{E:eta-asymp}, $\Big[ \partial_z |\eta|^2 \Big]_{-R}^R = O_{\lambda,h}(R^{2\sigma-2})$, $\int_{-R}^R |\eta_z|^2 = \int_{-\infty}^{+\infty} |\eta_z|^2 + O_{\lambda,h}(R^{2\sigma-2})$.  This gives
 \begin{equation}
 \label{E:Pohozhaev1}
  0 = \kappa \int_{-R}^R |\eta|^2 + h \Im \int_{-R}^R \Lambda \eta \; \bar \eta + \int_{-\infty}^{+\infty} |\eta_z|^2 - |\eta(0)|^{p+1} + O_{\lambda,h}(R^{2\sigma-2})
  \end{equation}
  
Take the imaginary part of \eqref{E:Poh1}, using that $\Im (\eta_z \bar \eta)_z = \Im ( \eta_{zz}\bar \eta)$ and $\Re \Lambda \eta \, \bar \eta = \frac12 ( z|\eta|^2)_z$, to obtain
$$
0 = h\sigma \int_{-R}^R |\eta|^2 - \frac12 h \int_{-R}^R \partial_z(z|\eta|^2) - \Im \int_{-R}^R \partial_z(\eta_z \bar \eta) $$
Applying the fundamental theorem on $(-R,0)$ and $(0,R)$, we obtain
$$ 0 = h\sigma \int_{-R}^R |\eta|^2 - \frac12 h \left[ z|\eta|^2 \right]_{-R}^R - \Im \left[ \eta_z \bar \eta \right]_{-R}^R + \Im (\eta_z(0+)-\eta_z(0-)) \overline{\eta(0)} $$
By \eqref{E:eta-juncture},
$$ 0 = h\sigma \int_{-R}^R |\eta|^2 - \frac12 h \left[ z|\eta|^2 \right]_{-R}^R - \Im \left[ \eta_z \bar \eta \right]_{-R}^R $$
By \eqref{E:eta-asymp}, $\eta_z \bar\eta \sim \overline{c_{0,\lambda,h}}c_{1,\lambda,h} |z|^{2\sigma-2}  \sgn z$, so $- \Big[ \Im \eta_z \bar \eta \Big]_{-R}^R = O_{\lambda,h} (R^{2\sigma-2})$.  Additionally from \eqref{E:eta-asymp}, we have $|\eta(z)|^2 \sim |c_{0,\lambda,h}|^2 |z|^{2\sigma-1}(1+O(z^{-2}))$, so  $-\frac12 h \Big[ z |\eta(z)|^2 \Big]_{-R}^R = -h |c_{0,\lambda,h}|^2 R^{2\sigma}+O(R^{2\sigma-2})$.  Hence
\begin{equation}
\label{E:Pohozhaev2}
0 = \sigma \int_{-R}^R |\eta|^2 - |c_{0,\lambda,h}|^2 R^{2\sigma} + O_{\lambda,h}(R^{2\sigma-2})
\end{equation}

Taking $R$ sufficiently large, we see from \eqref{E:Pohozhaev2} that $\sigma>0$ for a nontrivial solution.

Next multiply \eqref{E:eta-reform} by $\Lambda \bar \eta$,  and take the real part to obtain
$$\kappa \Re (\eta \Lambda \bar \eta) - h\sigma \Im (\eta \Lambda \bar\eta) - \Re(\eta_{zz} \Lambda \bar \eta) =0 \,, \qquad z\neq 0$$
Substituting the identities
$$\Re [ \eta \Lambda \bar \eta] = \frac12 \partial_z (z |\eta|^2)$$
$$ \Re [\eta_{zz} \Lambda \bar \eta] = \frac12 \Re (\eta_z \bar \eta)_z - |\eta_z|^2 + \frac12 (z|\eta_z|^2)_z$$
we obtain
$$\partial_z \left( \frac{\kappa}{2} z |\eta|^2 - \frac12 \Re(\eta_z \bar \eta) - \frac12 z |\eta_z|^2 \right) - h\sigma \Im (\eta\Lambda \bar \eta) + |\eta_z|^2 =0$$
Summing the integral over $-R<z<0$ and the integral over $0<z<R$, and using the fundamental theorem on each integral, we obtain
$$0=
\begin{aligned}[t]
& \Big[ \frac{\kappa}2  z|\eta|^2 - \frac12 \Re(\eta_z\bar\eta) - \frac12 z|\eta_z|^2\Big]_{-R}^R -h\sigma \Im \int_{-R}^R \eta \Lambda \bar \eta \\
&+ \int_{-R}^R |\eta_z|^2 + \frac12 \Re (\eta_z(0+)-\eta_z(0-)) \bar \eta(0) 
\end{aligned}
$$

By \eqref{E:eta-asymp}, $\eta_z \bar\eta \sim \overline{c_{0,\lambda,h}}c_{1,\lambda,h} |z|^{2\sigma-2}  \sgn z$, so $\Big[ \Re (\eta_z \bar \eta) \Big]_{-R}^R = O_{\lambda,h} (R^{2\sigma-2})$.  Additionally from \eqref{E:eta-asymp}, we have $z|\eta_z|^2 \sim |c_{1,\lambda,h}|^2 |z|^{2\sigma-2}(\sgn z)$, so $\Big[ z|\eta_z|^2\Big]_{-R}^R = O_{\lambda,h}(R^{2\sigma-2})$.  Also,  $z|\eta|^2 = |c_{0,\lambda,h}|^2 |z|^{2\sigma} (\sgn z)(1+O(z^{-2}))$, so $\Big[ z|\eta|^2 \Big]_{-R}^R  = 2|c_{0,\lambda,h}|^2 R^{2\sigma}+O_{\lambda,h}(R^{2\sigma-2})$

Substituting yields
\begin{equation}
\label{E:Pohozhaev3}
0 = \kappa |c_{0,\lambda,h}|^2 R^{2\sigma} - h\sigma \Im \int_{-R}^R \eta \Lambda \bar \eta + \int_{-R}^R |\eta_z|^2 - \frac12 |\eta(0)|^{p+1} + O(R^{2\sigma-2})
\end{equation}

From \eqref{E:Pohozhaev1}, we obtain
\begin{equation}
\label{E:Pohozhaev4}
\lim_{R\to +\infty} R^{-2\sigma} \left( \kappa \int_{-R}^R |\eta|^2 - h \Im \int_{-R}^R \eta \Lambda \bar \eta \right) =0
\end{equation}
From \eqref{E:Pohozhaev2}, we obtain
\begin{equation}
\label{E:Pohozhaev5}
\lim_{R\to +\infty} R^{-2\sigma} \sigma \int_{-R}^R |\eta|^2  = |c_{0,\lambda,h}|^2
\end{equation}
From \eqref{E:Pohozhaev3}, we obtain
\begin{equation}
\label{E:Pohozhaev6}
\lim_{R\to +\infty} R^{-2\sigma} h\sigma \Im \int_{-R}^R \eta \Lambda \bar \eta = \kappa |c_{0,\lambda,h}|^2
\end{equation}
By \eqref{E:Pohozhaev5}, we find that $\sigma>0$.  This combined with \eqref{E:Pohozhaev6} gives that 
$$\lim_{R\to +\infty} R^{-2\sigma} \Im \int_{-R}^R \eta \Lambda \bar \eta >0$$
This combined with \eqref{E:Pohozhaev4} gives that $\kappa>0$.

Plugging \eqref{E:Pohozhaev2} into \eqref{E:Pohozhaev3}, we obtain
\begin{equation}
\label{E:Pohozhaev7}
0 = \kappa \sigma \int_{-R}^R |\eta|^2 - h \sigma \Im \int_{-R}^R \eta \Lambda \bar \eta + \int_{-\infty}^{+\infty} |\eta_z|^2 - \frac12 |\eta(0)|^{p+1} +O_{\lambda,h}(R^{2\sigma-2})
\end{equation}
Taking \eqref{E:Pohozhaev7} and subtracting $\sigma$ times \eqref{E:Pohozhaev1} and sending $R\to \infty$, we obtain
$$0 = (1-\sigma) \int_{-\infty}^{+\infty} |\eta_z|^2 - \left( \frac12-\sigma\right) |\eta(0)|^{p+1}$$
Substituting,
$$E(\eta) = \frac12 \int |\eta_z|^2 - \frac{1}{p+1} |\eta(0)|^{p+1} = \left( \frac12 - \frac{2}{p+1}\cdot \frac{1-\sigma}{1-2\sigma}\right) \int |\eta_z|^2 $$
If $\sigma=\sigma_c=\frac12-\frac{1}{p-1}$, then note that $E(\eta)=0$.

\section{Existence and uniqueness of profiles}
\label{S:profiles}

In this section, we prove Theorem \ref{T:profiles-kappa+1}.    From the definition of $A(\lambda)$ given in \eqref{E:solve-condition}, and the definition of $w$ in terms of $v$ given in \eqref{E:w-wstar},
$$ A(\lambda) \defeq \frac{- v'(0,\lambda)}{v(0,\lambda)}$$
where $\lambda = -\kappa h^{-1}-i\sigma$, with $\kappa \in \{-1,1\}$.   From Theorem \ref{T:Pohozhaev}, we know that $0<\sigma<1$.    Substituting the explicit formulae given in \eqref{E:2.49} and \eqref{E:2.50}, we obtain
\begin{equation}
\label{E:A}
A(\lambda) = \frac{ e^{-i\pi/4} \sqrt{2} \Gamma( \frac34-\frac12i\lambda)}{\Gamma(\frac14-\frac12i\lambda)}
\end{equation}
Substituting $\lambda = -\kappa h^{-1}-i\sigma$, 
\begin{align*}
A(\lambda) &=  \frac{ e^{-i\pi/4} \sqrt{2} \Gamma( \frac34-\frac12\sigma +\frac12i\kappa h^{-1})}{\Gamma(\frac14-\frac12\sigma + \frac12i\kappa h^{-1})} \\
&= \frac{ e^{-i\pi/4} \sqrt{2} (-\frac14-\frac12\sigma+\frac12i\kappa h^{-1}) \Gamma(-\frac14 -\frac12\sigma + \frac12i \kappa h^{-1})}{\Gamma(\frac14-\frac12\sigma + \frac12 i \kappa h^{-1})} \\
& = \frac{ e^{-i\pi/4} \sqrt{2} (-\frac14-\frac12\sigma+\frac12i\kappa h^{-1}) \Gamma(-\frac14 -\frac12\sigma + \frac12i \kappa h^{-1})\Gamma(\frac14-\frac12\sigma-\frac12i\kappa h^{-1})}{|\Gamma(\frac14-\frac12\sigma + \frac12 i \kappa h^{-1})|^2}
\end{align*}
$A(\lambda)$ is real and positive if and only if $B(\lambda)$ is real and positive, where
\begin{equation}
\label{E:1-140}
\begin{aligned}
B(\lambda) &\defeq \frac{A(\lambda) |\Gamma(\frac14-\frac12\sigma+\frac12i\kappa h^{-1})|^2}{\sqrt{2}} \\
&= e^{-i\pi/4} (-\tfrac14-\tfrac12\sigma+\tfrac12i\kappa h^{-1}) \Gamma(-\tfrac14-\tfrac12\sigma+\tfrac12i\kappa h^{-1}) \Gamma(\tfrac14-\tfrac12\sigma-\tfrac12i\kappa h^{-1})
\end{aligned}
\end{equation}
For $\sigma=0$, we would use
\begin{equation}
\label{E:1-138}
-z\Gamma(-z)\Gamma(z) = \frac{\pi}{\sin \pi z}
\end{equation}
with $z = \frac14-\frac12i\kappa h^{-1}$ to simplify the expression for $B(\lambda)$.  Although $\sigma\neq 0$, it is expected to be exponentially small in $h$ as $h\to 0$, so we still use \eqref{E:1-138} with $z= \frac14+\frac12\sigma-\frac12i\kappa h^{-1}$ to obtain
\begin{equation}
\label{E:1-139}
\begin{aligned}
\indentalign (-\tfrac14 - \tfrac12 \sigma + \tfrac12 i \kappa h^{-1}) \Gamma(-\tfrac14 - \tfrac12 \sigma + \tfrac12 i \kappa h^{-1}) \\
&=-z\Gamma(-z) \\
&= \frac{\pi}{(\sin \pi z) \Gamma(z)}\\
&= \frac{\pi}{\sin \pi(\frac14+\frac12\sigma-\frac12i\kappa h^{-1}) \Gamma(\frac14+\frac12\sigma - \frac12i \kappa h^{-1})}
\end{aligned}
\end{equation}
Now
$$
\sin \pi(\tfrac14+\tfrac12\sigma-\tfrac12 i \kappa h^{-1}) = -i \frac12( e^{i\pi/4} e^{\frac12i\sigma\pi} e^{\frac12 \kappa h^{-1}\pi} - e^{-i\pi/4} e^{-\frac12i\sigma \pi} e^{-\frac12\kappa h^{-1}\pi}) $$
Hence
$$\sin \pi(\tfrac14+\tfrac12\sigma-\tfrac12 i \kappa h^{-1})  = 
\begin{cases}
\tfrac12 e^{-i\pi/4} e^{\frac12i\sigma\pi} e^{\pi/2h} ( 1+ i e^{-i\sigma \pi} e^{-\pi/h}) & \text{if }\kappa=+1\\
\tfrac12 e^{i\pi/4} e^{-\frac12i\sigma\pi} e^{\pi/2h} ( 1- i e^{i\sigma \pi} e^{-\pi/h}) & \text{if }\kappa=-1
\end{cases}
$$
Substituting into \eqref{E:1-139},
\begin{align*}
\indentalign (-\tfrac14 - \tfrac12 \sigma + \tfrac12 i \kappa h^{-1}) \Gamma(-\tfrac14 - \tfrac12 \sigma + \tfrac12 i \kappa h^{-1}) \\
&= \left\{\begin{aligned}
&\frac{ 2\pi e^{i\pi/4} e^{-\frac12 i \sigma \pi} e^{-\pi/2h}}{(1+ie^{-i\sigma \pi} e^{-\pi/h})\Gamma(\frac14+\frac12\sigma-\frac12i h^{-1})} && \text{if }\kappa=+1\\
&\frac{ 2\pi e^{-i\pi/4} e^{\frac12 i \sigma \pi} e^{-\pi/2h}}{(1-ie^{i\sigma \pi} e^{-\pi/h})\Gamma(\frac14+\frac12\sigma+\frac12i h^{-1})} && \text{if }\kappa=-1
\end{aligned} \right.
\end{align*}
Substituting into \eqref{E:1-140},
$$
B(\lambda) = \left\{ \begin{aligned} \frac{ 2\pi e^{-\frac12i\sigma \pi} e^{-\pi/2h}}{1+ i e^{-i\sigma \pi} e^{-\pi/h}} \frac{ \Gamma( \frac14 - \frac12 \sigma - \frac12 i h^{-1})}{\Gamma( \frac14 + \frac12\sigma -\frac12 i h^{-1})}  && \text{if }\kappa=+1\\
\frac{ -2\pi i e^{\frac12 i \sigma \pi} e^{-\pi/2h}}{1-ie^{i\sigma \pi} e^{-\pi/h}}\frac{\Gamma(\frac14-\frac12\sigma+\frac12ih^{-1})}{\Gamma(\frac14+\frac12\sigma+\frac12i h^{-1})} && \text{if }\kappa=-1
\end{aligned}\right.
$$
Combining with \eqref{E:1-140}, we have
\begin{equation}
\label{E:B}
A(\lambda) = \frac{\sqrt{2}}{|\Gamma(\frac14-\frac12\sigma+\frac12i\kappa h^{-1})|^2} \left\{ \begin{aligned} \frac{ 2\pi e^{-\frac12i\sigma \pi} e^{-\pi/2h}}{1+ i e^{-i\sigma \pi} e^{-\pi/h}} \frac{ \Gamma( \frac14 - \frac12 \sigma - \frac12 i h^{-1})}{\Gamma( \frac14 + \frac12\sigma -\frac12 i h^{-1})}  && \text{if }\kappa=+1\\
\frac{ -2\pi i e^{\frac12 i \sigma \pi} e^{-\pi/2h}}{1-ie^{i\sigma \pi} e^{-\pi/h}}\frac{\Gamma(\frac14-\frac12\sigma+\frac12ih^{-1})}{\Gamma(\frac14+\frac12\sigma+\frac12i h^{-1})} && \text{if }\kappa=-1
\end{aligned}\right.
\end{equation}
Thus we have two expressions for $A(\lambda)$ given by \eqref{E:A} and \eqref{E:B}. 
Now let 
\begin{equation}
\label{E:f-def}
f(\sigma,h^{-1}) \defeq \Im \log A(\lambda)
\end{equation}
In either case $\kappa=\pm 1$,  there exists an outgoing solution $\varphi$ if and only if $f(\sigma,h^{-1}) \in 2\pi \mathbb{Z}$.   From \eqref{E:f-def} and \eqref{E:A}, we have
\begin{equation}
\label{E:f-1}
f(\sigma,h^{-1}) = -\tfrac14\pi + \Im \log \Gamma( \tfrac34-\tfrac12\sigma+\tfrac12i\kappa h^{-1}) - \Im \log \Gamma(\tfrac14-\tfrac12\sigma + \tfrac12i \kappa h^{-1})
\end{equation}
From \eqref{E:f-def} and \eqref{E:B}, we have
\begin{equation}
\label{E:1-142}
\begin{aligned}
f(\sigma, h^{-1}) = 
\begin{cases}
-\tfrac12 \sigma \pi  - \Im \log(1+ ie^{-i\sigma \pi}e^{-\pi/h}) \\
+\Im \log \Gamma( \tfrac14 - \tfrac12 \sigma - \tfrac12 i h^{-1}) - \Im \log \Gamma( \tfrac14 + \tfrac12 \sigma - \tfrac12i h^{-1}) & \text{if }\kappa=+1 \\
\;\\
-\frac{\pi}{2} +\tfrac12 \sigma \pi  - \Im \log(1- ie^{i\sigma \pi}e^{-\pi/h}) \\
+\Im \log \Gamma( \tfrac14 - \tfrac12 \sigma + \tfrac12 i h^{-1}) - \Im \log \Gamma( \tfrac14 + \tfrac12 \sigma + \tfrac12i h^{-1}) & \text{if }\kappa=-1 
\end{cases}
\end{aligned}
\end{equation}

Let us comment on the branches of the logarithm in \eqref{E:f-1} and \eqref{E:1-142}.  The terms involving $\log \Gamma(z)$ for certain $z$ are defined as follows.  By the Weierstrass product representation, $\Gamma(z)$ is analytic on $\mathbb{C} \backslash \{0,-1,-2, \ldots \}$ and \emph{nonvanishing}, with poles at $0,-1,-2, \ldots$.    We thus restrict $\log \Gamma(z)$ to the simply connected domain $\mathbb{C} \backslash \text{(negative real axis)}$, and fix it to be the analytic continuation that results from assigning $\log \Gamma(1) = 0$.   If we restrict to $0<\sigma<\frac12$, then each input value $z$ of $\log \Gamma(z)$ appearing in \eqref{E:f-1} and \eqref{E:1-142} belongs to $\mathbb{C} \backslash \text{(negative real axis)}$, allowing for $h^{-1}=0$.   In the case $\frac12 \leq \sigma <1$, we restrict to $h^{-1} >0$ but can still assign values to $\log \Gamma(z)$ for $h^{-1}=0$ for each input $z$ in   \eqref{E:f-1} and \eqref{E:1-142} by taking the limit $h^{-1} \searrow 0$.   In \eqref{E:1-142}, there is an additional term with a logarithm.  Since $|ie^{-i\sigma \pi}e^{-\pi/h}|<1$, we have $\Re( 1\pm i e^{-i\sigma \pi} e^{-\pi/h})>0$, and the function $\log(1+ ie^{-i\sigma \pi}e^{-\pi/h})$ is taken as the branch of $\log(w)$ such that $\log w$ is real for $w$ real.

\begin{lemma}
\label{L:profiles-kappa+1}
For $0<\sigma<1$, $h>0$, $\kappa=1$, we have $-\frac{\pi}{2}<f(0,h^{-1})<0$ and $0<f(1, h^{-1})<\pi$ for all $0<h^{-1}<\infty$, and $\partial_\sigma f(\sigma, h^{-1})>0$ for all $0<\sigma<1$ and $0<h^{-1}<\infty$.  Hence for each $0<h^{-1}<\infty$, there exists a unique $0<\sigma(h)<1$ such that $f(\sigma(h), h^{-1})=0$ and there are no solutions to $f(\sigma,h^{-1})=2\pi n$ for $n\neq 0$.
\end{lemma}

\begin{proof}
From \eqref{E:1-142}, we have $f(0,h^{-1}) = -\Im \log(1+ie^{-\pi/h})$, from which it readily follows that $-\frac{\pi}{2} < f(0, h^{-1})< 0$.

 We now turn to evaluating $f(1, h^{-1})$.  Consider the formula
\begin{align}
\indentalign \log \Gamma( \tfrac14+\tfrac12\sigma - \tfrac12 i h^{-1}) = \log ( ( -\tfrac34 + \tfrac12 \sigma - \tfrac12 i h^{-1})\Gamma(-\tfrac34+\tfrac12\sigma-\tfrac12i h^{-1})) 
\notag \\
& = \log ( -\tfrac34 + \tfrac12 \sigma - \tfrac12 i h^{-1}) + \log \Gamma(-\tfrac34+\tfrac12\sigma-\tfrac12i h^{-1})
\label{E:gprod}
\end{align}
in the simply connected domain $(\sigma, h^{-1}) \in \mathbb{R}^2$ excluding $(-\infty, \frac32)\times \{0\}$, but we must specify the branch of $\log ( -\frac34 + \frac12 \sigma - \frac12 i h^{-1})$.    If we consider the point $(\sigma, h^{-1}) = (\frac72, 0)$, then the left side is $\log \Gamma( \frac14+\frac12\sigma - \frac12 i h^{-1})= \log \Gamma(2) = \log 1 =0$, and on the right side, $\log \Gamma(-\frac34+\frac12\sigma-\frac12i h^{-1}) = \log\Gamma(1) = 0$, so we need to take $\log ( -\frac34 + \frac12 \sigma - \frac12 i h^{-1}) = \log 1 =0$ (as opposed to other integer multiples of $2\pi i$).  Hence the branch of $\log ( -\frac34 + \frac12 \sigma - \frac12 i h^{-1})$ in the above formula has imaginary part ranging from $-\pi$ to $+\pi$ (note that  $-\frac34 + \frac12 \sigma - \frac12 i h^{-1}$ on the specified domain takes values in $\mathbb{C}\backslash \text{(negative reals)}$).

Evaluating \eqref{E:gprod} at $\sigma=1$ gives
$$\log\Gamma(\tfrac34-\tfrac12ih^{-1}) = \log( -\tfrac14 - \tfrac12 ih^{-1}) + \log\Gamma(-\tfrac14 - \tfrac12 ih^{-1})$$
Substituting into \eqref{E:1-142} gives
$$f(1,h^{-1}) = -\tfrac12\pi - \Im \log(1-ie^{-\pi/h}) - \Im \log(-\tfrac14-\tfrac12 i h^{-1})$$
Now
$$-\tfrac12\pi < \Im \log(1- ie^{-\pi/h}) < 0 \,, \qquad -\pi < \log( -\tfrac14 - \tfrac12 i h^{-1}) < -\tfrac12\pi$$
so $0< f(1,h^{-1}) < \pi$ as claimed.  

Next, we prove that $\partial_\sigma f(\sigma, h^{-1})>0$.  From \eqref{E:f-1},
\begin{equation}
\label{E:partial-sigma-f}
\begin{aligned}
\partial_\sigma f(\sigma, h^{-1}) &= \tfrac12 \Im  \left( -\psi( \tfrac34-\tfrac12\sigma + \tfrac12ih^{-1})  + \psi( \tfrac14 - \tfrac12\sigma + \tfrac12i h^{-1}) \right) \\
&= \tfrac12 \Im( \psi(z) - \psi(z+\tfrac12))
\end{aligned}
\end{equation}
where $\psi$ is the digamma function and we take $z= \frac14 - \frac12\sigma+\frac12ih^{-1}$.  By taking the log derivative of the Weierstrass product formula for the gamma function, we obtain
$$\psi(z) = -\gamma + \sum_{k=0}^{\infty} \left( \frac{1}{k+1} - \frac{1}{z+k}\right)$$
Hence
$$\psi(z)-\psi(z+\tfrac12) = \sum_{k=0}^{+\infty} \left( -\frac{1}{z+k} +\frac{1}{z+\frac12+k}\right)= -\frac12 \sum_{k=0}^\infty \frac{1}{(z+k)(z+\frac12+k)}$$
For $z=x+iy$, we obtain
$$\Im (\psi(z)-\psi(z+\frac12)) =  \frac12 \sum_{k=0}^\infty \frac{y(2x+\frac12+2k)}{[ (x+k)(x+\frac12+k)-y^2]^2+[ y(2x+\frac12+2k)]^2}$$
In our case $z=\frac14-\frac12\sigma+\frac12ih^{-1}$, so $x=\frac14-\frac12\sigma$ and $y=\frac12h^{-1}$.  Hence
$$y(2x+\tfrac12+2k) = \tfrac12 h^{-1}( 1-\sigma+2k) >0$$
for all $0<h^{-1}<\infty$ and $0<\sigma<1$.  Consequently, from \eqref{E:partial-sigma-f}, we obtain $\partial_\sigma f(\sigma, h^{-1})>0$, as claimed.
\end{proof}

\begin{lemma}
\label{L:profiles-kappa-1}
For $0<\sigma<1$, $h>0$, $\kappa=-1$, we have $-\frac{\pi}{2}<f(0,h^{-1})<0$ and $-\frac{3\pi}{2}<f(1, h^{-1})<-\frac{\pi}{2}$ for all $0<h^{-1}<\infty$, and $\partial_\sigma f(\sigma, h^{-1})<0$ for all $0<\sigma<1$ and $0<h^{-1}<\infty$.  Hence there are no solutions to $f(\sigma, h^{-1}) \in 2\pi \mathbb{Z}$ for $0<\sigma<1$ and $0<h^{-1}<\infty$.
\end{lemma}
\begin{proof}
The proof is analogous to the proof of Lemma \ref{L:profiles-kappa+1} and will be omitted.
\end{proof}

\section{Asymptotic calculation of $\sigma(h)$}
\label{S:profiles-asymptotic}

In this section we prove Theorem \ref{T:profiles-asymptotic}, starting with the case $0<h\ll 1$.  We use
\begin{align}
\notag
\log( 1+ i e^{-i\sigma \pi} e^{-\pi/h}) &= i e^{-i\sigma \pi} e^{-\pi/h} + O(e^{-2\pi/h}) \\
\label{E:exp-10}
&= i e^{-\pi/h} + O(e^{-2\pi/h}) + O(\sigma e^{-\pi/h})
\end{align}
and we also use the expansion
\begin{lemma}
\label{L:gamma-difference}
$$\log \Gamma( z+ \sigma) - \log \Gamma(z) = \sigma \log z  - \tfrac12 \sigma (1-\sigma) z^{-1} +  O( \sigma |z|^{-2})$$
\end{lemma}

We will prove this lemma below.   Let  $z = \frac14 - \frac12\sigma -\frac12 i h^{-1} = -\frac12 i h^{-1}(1+\frac12 i h - i \sigma h)$.  Then $z^{-1} = 2ih+ O(h^2)$ and
\begin{align*}
\sigma \log z &= \sigma \log [(-i) (\tfrac12 h^{-1})(1+\tfrac12 i h - i \sigma h)] \\
& = \sigma ( - \tfrac12\pi i + \log (\tfrac12 h^{-1}) + \log( 1+\tfrac12 i h - i\sigma h)) \\
& = -\tfrac12\pi \sigma i + \sigma \log(\tfrac12 h^{-1}) + \tfrac12i\sigma h - i \sigma^2 h + O(\sigma h^2)
\end{align*}
Also,
$$- \tfrac12 \sigma(1-\sigma) z^{-1} = -i\sigma(1-\sigma) h + O(\sigma h^2)$$
Plugging into Lemma \ref{L:gamma-difference}, we obtain
\begin{align}
\notag \indentalign \log \Gamma(\tfrac14 - \tfrac12\sigma -\tfrac12i h^{-1}) - \log \Gamma(\tfrac14 + \tfrac12\sigma - \tfrac12 i h^{-1}) \\
\label{E:exp-11}
&= - \sigma \log (\tfrac12 h^{-1}) + \tfrac12 i\pi \sigma + \tfrac12 i \sigma h  + O(\sigma h^2)
\end{align}
Substituting the asymptotic expansions \eqref{E:exp-10}, \eqref{E:exp-11} (valid for $0<\sigma<1$) into \eqref{E:1-142}, we obtain
$$
f(\sigma, h^{-1}) =  \tfrac12 \sigma h - e^{-\pi/h} + O(e^{-2\pi/h}) + O(\sigma h^2)
$$
Note that the ``big $O$'' notation means that there exists $C>0$ such that for every $0<\sigma<1$, \begin{equation}
\label{E:log-B-expand-2}
\tfrac12\sigma h (1-Ch) - e^{-\pi/h}(1+Ce^{-\pi/h}) \leq f(\sigma) \leq \tfrac12 \sigma h(1+Ch) - e^{-\pi/h}(1-Ce^{-\pi/h})
\end{equation}
If $0<\sigma \leq 2h^{-1}e^{-\pi/h}(1-Ch)$, then we have from \eqref{E:log-B-expand-2} that
\begin{align*}
f(\sigma) &\leq \tfrac12 \sigma h ( 1+Ch) - e^{-\pi/h}(1- Ce^{-\pi/h})\\
&\leq  (1-Ch)(1+Ch) e^{-\pi/h} - e^{-\pi/h}(1-Ce^{-\pi/h}) \\
&= (-C^2h^2 + Ce^{-\pi/h}) e^{-\pi/h} \\
&<0
\end{align*}
If $2h^{-1}e^{-\pi/h}(1+2Ch) \leq \sigma <1$, then we have from \eqref{E:log-B-expand-2} that
\begin{align*}
f(\sigma) &\geq \tfrac12 \sigma h(1-Ch) - e^{-\pi/h}(1+Ce^{-\pi/h}) \\
&\geq (1+2Ch)(1-Ch) e^{-\pi/h} - e^{-\pi/h}(1+Ce^{-\pi/h}) \\
& \geq (1+Ch - 2C^2h^2-Ce^{-\pi/h}) e^{-\pi/h} \\
&>0
\end{align*}
Hence the unique solution to $f(\sigma, h^{-1})=0$ lies in the interval
$$2h^{-1}e^{-\pi/h}(1-Ch) e^{-\pi/h} \leq \sigma \leq 2h^{-1}e^{-\pi/h}(1+2Ch)$$

\begin{proof}[Proof of Lemma \ref{L:gamma-difference}]
Recall the Binet's log Gamma formula
\begin{equation}
\label{E:Binet}
\log \Gamma(z) = (z-\tfrac12)\log z - z + \tfrac12 \ln(2\pi) + 2 \int_0^{+\infty} \frac{\arctan(tz^{-1})}{e^{2\pi t}-1} \, dt
\end{equation}
By the mean-value theorem (using that $|(\arctan z)'| \leq 1$) we obtain
$$|\arctan(t(z+\sigma)^{-1}) - \arctan(tz^{-1})| \leq t\left| \frac{1}{z+\sigma} - \frac{1}{z} \right| \leq \frac{t\sigma}{|z(z+\sigma)|}$$
and hence
$$\left| \int_0^{+\infty} \frac{\arctan(t(z+\sigma)^{-1}) - \arctan(tz^{-1})}{e^{2\pi t}-1} \, dt \right| \leq 2\sigma |z|^{-2} \int_0^\infty \frac{t}{e^{2\pi t}-1} \, dt $$ 
it follows that
\begin{align}
\notag \indentalign \log \Gamma(z+\sigma) - \log \Gamma(z) \\
\notag &= (z+\sigma - \tfrac12) \log(z+\sigma) - (z-\tfrac12)\log z - \sigma + O(\sigma |z|^{-2}) \\
\notag &= (z-\tfrac12 + \sigma) \log[ z(1+\sigma z^{-1})] - (z-\tfrac12)\log z - \sigma + O(\sigma |z|^{-2}) \\
\label{E:log-exp-1} &= \sigma\log z + (z-\tfrac12 + \sigma) \log (1+\sigma z^{-1}) - \sigma + O (\sigma |z|^{-2})
\end{align}
Now also
\begin{align}
\notag (z-\tfrac12 + \sigma) \log (1+\sigma z^{-1}) &= (z-\tfrac12+\sigma)( \sigma z^{-1} - \tfrac12 \sigma^2 z^{-2} + O(\sigma^3z^{-3})) \\
\label{E:log-exp-2}
&= \sigma - \tfrac12 \sigma(1-\sigma) z^{-1}
\end{align}
Substituting \eqref{E:log-exp-2} into \eqref{E:log-exp-1}, we obtain the claimed expansion.
\end{proof}

Now we return to compute the amplitude of the outgoing solution $\varphi$ at $x=0$.  Recall
$$\varphi(x) = \alpha v(h^{1/2}|x|)\,, \qquad \alpha = \frac{ (2h^{1/2} A(\lambda))^{1/(p-1)}}{v(0)} $$
with
$$A(\lambda) = -\frac{v'(0)}{v(0)}=\frac{ e^{-i\pi/4} \sqrt{2} \Gamma( \frac34-\frac12\sigma +\frac12ih^{-1})}{\Gamma(\frac14-\frac12\sigma +\frac12ih^{-1})}$$
 We have chosen $\sigma = 2e^{-\pi/h}h^{-1}(1+O(h))$ precisely so that $A>0$, and thus $A^{1/(p-1)}$ is well-defined. 
From \url{http://dlmf.nist.gov/5.11#E14}, we have
$$\frac{\Gamma(z+a)}{\Gamma(z+b)} = z^{a-b}(1+O(z^{-1}))$$
Taking $z = \frac12 i h^{-1}$, $a=\frac34-\frac12\sigma$, $b=\frac14-\frac12\sigma$, we obtain
$$A = h^{-1/2}(1+O(h))$$
hence
\begin{equation}
\label{E:alpha1}
\alpha = \frac{ 2^{1/(p-1)}(1+O(h))}{v(0)}
\end{equation} 
Hence
$$\varphi(0) = \alpha v(0) = 2^{1/(p-1)} (1+O(h))$$
To compute that asymptotic behavior of $\varphi(x)$ as $|x|\to \infty$, we need to compute $\alpha$ and hence $v(0)$.  In \S\ref{S:Weber}, we show that
$$v(0) = \frac{\sqrt{\pi} 2^{\frac12 i \lambda - \frac14}}{\Gamma(\frac34-\frac12i\lambda)}$$
With $\lambda = -ih^{-1} -i\sigma$, this becomes
$$v(0) = \frac{\sqrt{\pi} e^{-\frac12 i h^{-1} \log 2} 2^{-\frac14 +\frac12\sigma}}{\Gamma(\frac34-\frac12\sigma+\frac12 i h^{-1})}$$
Using asymptotic expansion in \eqref{E:Binet}, we obtain that as $h \to 0$, $\sigma \approx 2h^{-1}e^{-\pi/h}$, 
$$ \Gamma(\tfrac12ih^{-1}+\tfrac34-\tfrac12\sigma) \approx (2\pi)^{1/2}  2^{-1/4}e^{\pi i/8} h^{-1/4} e^{-\frac{\pi}{4h}} e^{i\frac12h^{-1}\log (\frac12 h^{-1})} e^{-i\frac12h^{-1}}$$
and hence
$$v(0) \approx 2^{-1/2} e^{-\pi i/8} h^{1/4} e^{\pi/4h} e^{-\frac12i h^{-1}\log h^{-1}} e^{\frac12 i h^{-1}}$$
From \eqref{E:alpha1} it follows that
\begin{equation}
\label{E:alpha-asymp}
\alpha \approx 2^{1/(p-1)}2^{1/2} e^{\pi i/8} h^{-1/4} e^{-\pi/4h} e^{\frac12i h^{-1}\log h^{-1}} e^{-\frac12 i h^{-1}}
\end{equation}

Finally, we compute $\sigma(h)$ in the case $h\gg 1$ (i.e. $0<h^{-1}\ll 1$) and obtain the formula presented in the second part of Theorem \ref{T:profiles-asymptotic}.  Taking $z=\frac14-\frac12\sigma+\frac12ih^{-1}$, we obtain from \eqref{E:A} that
\begin{equation}
\label{E:A2}
A(\lambda) = \frac{e^{-i\pi/4} \sqrt{2}\Gamma(\frac12+z)}{\Gamma(z)}
\end{equation}
In the case $h^{-1}\to 0$, we must have $\sigma \to \frac12$, for otherwise if $\sigma\to \sigma_0\neq \frac12$, then $\Gamma(\frac12+z)\to \Gamma(\frac34-\frac12\sigma_0)$ and $\Gamma(z)\to \Gamma(\frac14-\frac12\sigma_0)$, both of which are real (and finite).  Hence the right side of \eqref{E:A2} cannot converge to a real and positive value.    Given that $\sigma\to \frac12$ as $h^{-1}\to 0$, we have that $z\to 0$.  Hence we reexpress the denominator of \eqref{E:A2} as $z^{-1}\Gamma(1+z)$ to obtain
\begin{equation}
\label{E:A3}
A(\lambda) = \frac{e^{-i\pi/4} \sqrt{2} z\Gamma(\frac12+z)}{\Gamma(1+z)} 
\end{equation}
Since $\Gamma(\frac12+z)\to \Gamma(\frac12)=\sqrt{\pi}$ and $\Gamma(1+z)\to \Gamma(1)=1$, we conclude that $\arg(z) \to \frac{\pi}{4}$, i.e. $\Re(z) \approx \Im(z)$ as $h^{-1}\to 0$.  This implies $\frac14-\frac12\sigma \approx \frac12 h^{-1}$, or $\sigma \approx \frac12-h^{-1}$ as $h^{-1}\to 0$.

\section{Form of outgoing profiles as $h\to 0$}
\label{S:WKB}

In this section, we prove Theorem \ref{T:WKB}.  This follows from the calculation of $\alpha$ in \eqref{E:alpha-asymp} and the following lemma.

\begin{figure}
\includegraphics[scale=0.5]{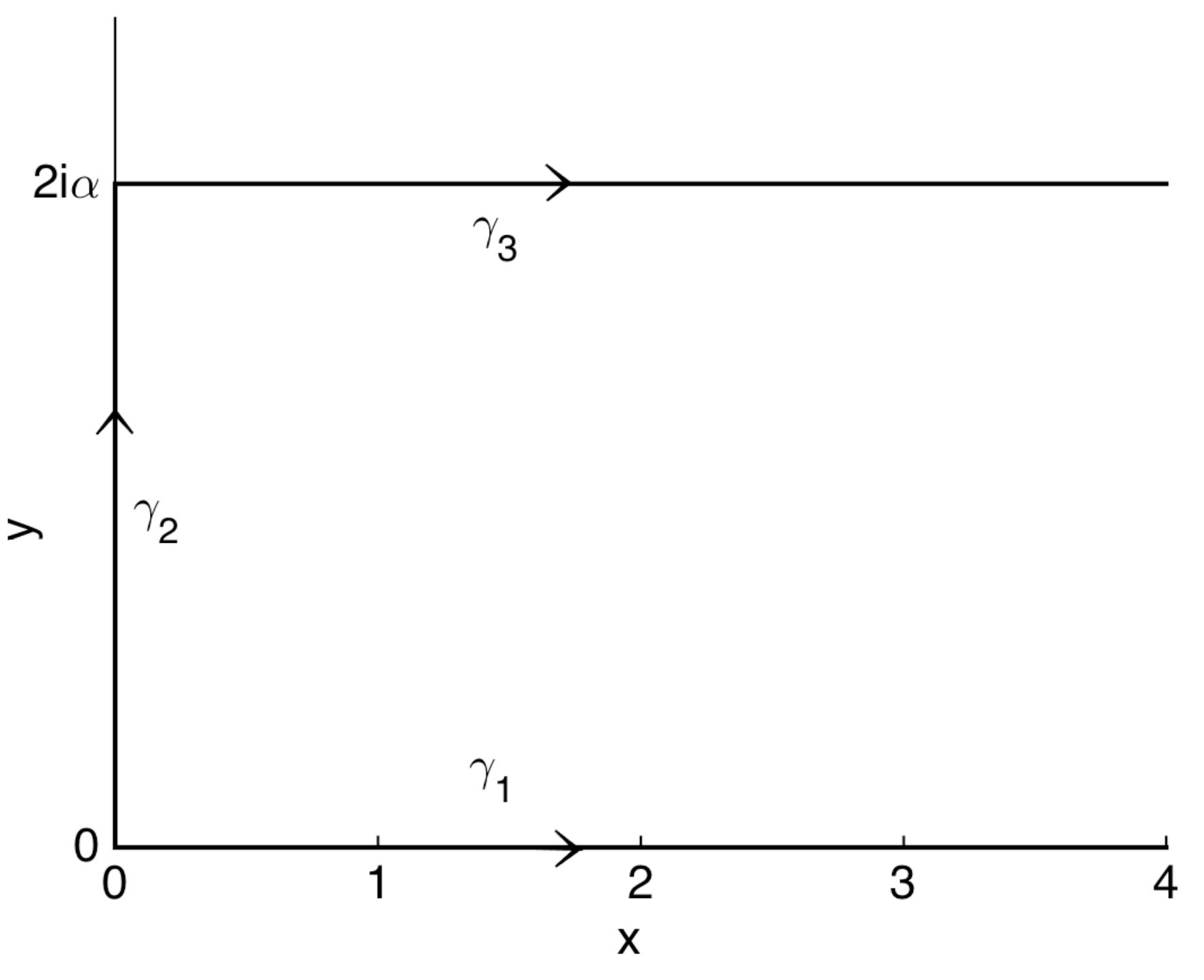}
\caption{\label{F:WKB-3}Depiction of the curves $\gamma_1$, from $0$ to $+\infty$ along the positive real axis, $\gamma_2$, from $0$ to $2 i\alpha^2$ along the positive imaginary axis, and $\gamma_3$, the curve in the first quadrant following the curve $z=r(\theta)e^{i\theta}$, where $r(\theta) = \frac12 \csc \theta$.}
\end{figure}

\begin{figure}
\includegraphics[scale=0.7]{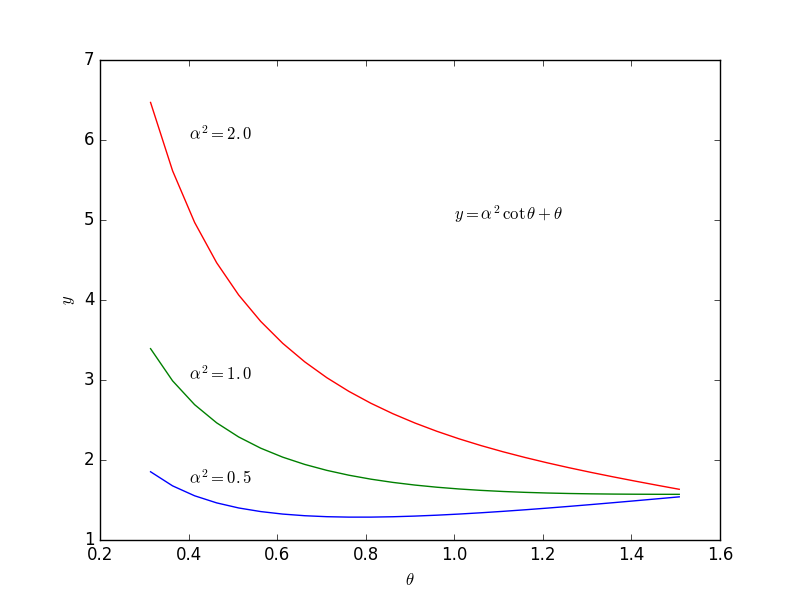}
\caption{\label{F:WKB-2}Graphs of $f_\alpha(\theta) = \alpha^2\cot \theta + \theta$ on $0<\theta\leq \frac{\pi}{2}$ for $\alpha^2=2.0$, $1.0$, and $0.5$.  For each $\alpha$, we have $f(\frac{\pi}{2})=\frac{\pi}{2}$ and $f(\theta) \sim \alpha^2 \theta^{-1}$ as $\theta \searrow 0$.  For $\alpha\geq 1$, $f(\theta)$ is decreasing on the whole interval, but for $0<\alpha<1$, $f(\theta)$ achieves a minimum in the middle at $\theta_0 = \arcsin \alpha$ with value $f(\theta_0) = \alpha\sqrt{1-\alpha^2} + \arcsin\alpha$.}
\end{figure}

\begin{figure}
\includegraphics[scale=0.7]{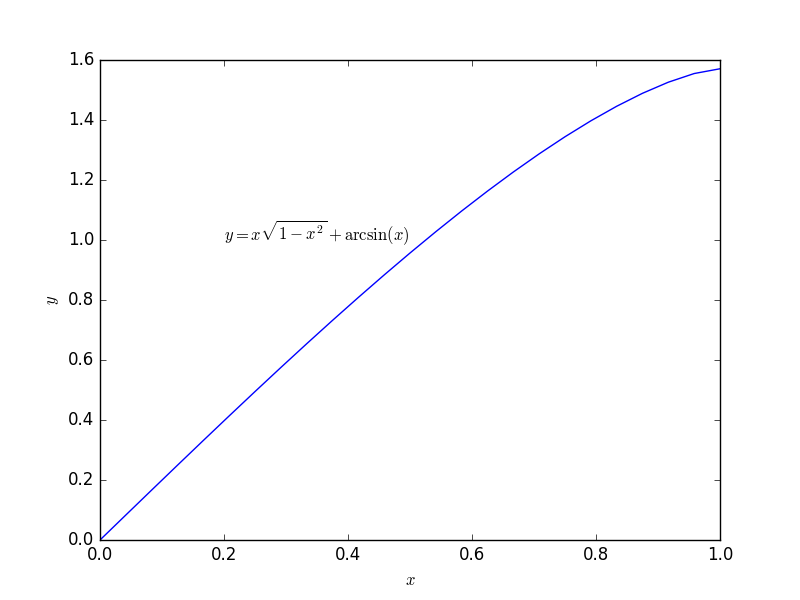}
\caption{\label{F:WKB-1}Graph of $S(x) = x\sqrt{1-x^2}+\arcsin x$.}
\end{figure}

\begin{lemma}[$x\to+\infty$ asymptotics of $v(x)$]
\label{L:asymp}
For $\lambda = -h^{-1}-i\sigma$, $0<h \ll 1$, and $\sigma \approx 2h^{-1}e^{-\pi/h}$, defining $\alpha \defeq \frac12 h^{1/2}x$, as $h\to 0$, we have the expansion
$$ v(x) \approx 
\begin{cases}
 (1-\alpha)^{-1/4} e^{-h^{-1}[\arcsin\alpha + \alpha(1-\alpha^2)^{1/2}]} & \text{for }x\ll 2h^{-1/2} \\
 e^{\frac14 ix^2} x^{-ih^{-1}+\sigma-\frac12} & \text{for }x\gg 2h^{-1/2}
\end{cases}
$$
\end{lemma}

\begin{proof}
Recall we restrict to $x>0$.  As $x=2h^{-1/2}$ is the classical turning point, it is convenient to use the parameter $\alpha \defeq \frac12 h^{1/2}x$, which has the property that $\alpha=1$ exactly when $x=2h^{-1/2}$.

By definition, we have
\begin{equation}
\label{EE:2.5-2}
v(x) = \frac{e^{\frac14 ix^2}}{\Gamma(ih^{-1}-\sigma+\frac12)} \int_0^\infty e^{-\frac12 t^2} e^{- \big( \frac{1-i}{\sqrt{2}}\big) tx} \, t^{ih^{-1}-\sigma-\frac12} \, dt 
\end{equation}
for $\sigma<\frac12$.    By changing variables $t\mapsto h^{-1} x^{-1}t$, we obtain
\begin{equation}
\label{EE:2.5-3}
v(x) = e^{\frac14 ix^2}  ( hx)^{-ih^{-1}+\sigma-\frac12} \frac{1 }{\Gamma(ih^{-1}-\sigma+\frac12)} \int_0^\infty e^{-\frac12 h^{-2} x^{-2} t^2} e^{- \big( \frac{1-i}{\sqrt{2}}\big) h^{-1} t } \, t^{ih^{-1}-\sigma-\frac12} \, dt 
\end{equation}
By rotating contour forward by $e^{\pi i/4}$, we obtain
\begin{equation}
\label{EE:2.5-4}
v(x) = e^{\frac14 ix^2}  (hx)^{-ih^{-1}+\sigma-\frac12} \frac{ e^{-\pi/(4h)}  e^{-i\sigma \pi/4} e^{\pi i/8}  }{\Gamma(ih^{-1}-\sigma+\frac12)} g(\alpha, h^{-1})
\end{equation}
where
$$g(\alpha, h^{-1}) \defeq \int_0^\infty e^{-h^{-1}(\frac18 i \alpha^{-2}t^2 + t - i \log t)}t^{-\sigma-\frac12}  \, dt $$
Note that the application of Cauchy's theorem required to deduce \eqref{EE:2.5-4} from \eqref{EE:2.5-3} is straightforward since the functions in the exponential have negative real part.    Thus one can use the standard wedge contour, and we will not further elaborate on this calculation.  Taking
$$p_\alpha(z) = -\tfrac18 i \alpha^{-2}z^2 - z +i \log z$$
then
\begin{equation}
\label{EE:2.5-5}
g(\alpha, h^{-1}) = \int_{\gamma_1} e^{h^{-1} p_\alpha(z)} z^{-\sigma- \frac12} \, dz
\end{equation}
where $\gamma_1$ denotes the positive real axis, oriented from $0$ to $+\infty$.  We would like to rotate forward the contour $\gamma_1$ in \eqref{EE:2.5-5} from the positive real axis to the positive imaginary axis, although this requires moving through a region where $-\frac18 i \alpha^{-2} z^2$ has \emph{positive} real part.  Taking $z=re^{i\theta}$, we have
$$\Re p_\alpha(z) = \tfrac18 \alpha^{-2} r^2 \sin 2\theta - r\cos\theta - \theta$$
\begin{equation}
\label{E:Im-p}
\Im p_\alpha(z) = -\tfrac18 \alpha^{-2} r^2 \cos 2\theta - r\sin \theta + \log r
\end{equation}
Using the identity $\sin 2\theta = 2\sin \theta \cos \theta$ and completing the square, we have
$$\Re p_\alpha(z) = \tfrac14 \alpha^{-2} \sin \theta \cos \theta ( r-  2\alpha^2 \csc\theta)^2 - f_\alpha(\theta)$$
where
 $$f_\alpha(\theta) \defeq \alpha^2\cot \theta + \theta$$
This suggests to deform along the contour $r =2 \alpha^2 \csc\theta$, denoted $\gamma_2$, to link up with the segment $\gamma_3$ of the positive imaginary axis from $0$ to $+2i\alpha^2$, as in Fig. \ref{F:WKB-3}.    Cauchy's theorem implies 
$$g(\alpha, h^{-1}) = g_2(\alpha, h^{-1}) + g_3(\alpha, h^{-1})$$
where
$$g_j(\alpha, h^{-1}) = \int_{\gamma_j} e^{h^{-1} p_\alpha(z)} z^{-\sigma- \frac12} \, dz$$

First we shall examine $g_2(\alpha, h^{-1})$.  We parameterize the contour $\gamma_2$ as $z = r(\theta)e^{i\theta}$, with $r(\theta) = 2\alpha^2 \csc\theta$ from $\theta=\frac{\pi}{2}$ to $\theta = 0$, and have
$$dz = (r'(\theta)+ir(\theta)) e^{i\theta} d\theta = 2\alpha^2 \csc^2\theta d\theta$$
Thus
$$g_2(\alpha, h^{-1}) = -(2\alpha^2)^{\frac12-\sigma} \int_0^{\pi/2} e^{-h^{-1} f_\alpha(\theta)} e^{i h^{-1} \nu_\alpha(\theta)} e^{i (\sigma+\frac12)\theta} (\sin \theta)^{\sigma-\frac32}  \,  d\theta $$
where
$$\nu_\alpha(\theta) \defeq \Im p_\alpha(r(\theta)e^{i\theta})$$
Fig. \ref{F:WKB-2} shows a plot of $f_\alpha(\theta)$.    For each $\alpha$, we have $f_\alpha(\frac{\pi}{2})=\frac{\pi}{2}$ and $f(\theta) \sim \alpha^2 \theta^{-1}$ as $\theta \searrow 0$.  For $\alpha\geq 1$, $f(\theta)$ is decreasing on the whole interval, but for $0<\alpha<1$, $f(\theta)$ achieves a minimum in the middle at $\theta_0 = \arcsin \alpha$ with value 
\begin{equation}
\label{E:SP1}
f_\alpha(\theta_0) = \alpha(1-\alpha^2)^{1/2} + \arcsin\alpha \,, \qquad f_\alpha''(\theta_0) = 2\alpha^2 \frac{\cos\theta_0}{\sin^3\theta_0}
\end{equation}
Use of the double angle identity $\cos2\theta = 1-2\sin^2\theta$ and \eqref{E:Im-p} gives
\begin{equation}
\label{E:SP2}
\nu_\alpha(\theta) = -\tfrac12 \alpha^2 \csc^2\theta + \log \csc\theta  - \alpha^2 +\log(2\alpha^2)
\end{equation}
We also find that when $\alpha<1$, with $\theta_0=\arcsin\alpha$, that $\nu_\alpha'(\theta_0)=0$ and
\begin{equation}
\label{E:SP3}
\nu_\alpha(\theta_0) = -\tfrac12 - \alpha^2 + \log(2\alpha) \,, \qquad \nu_\alpha''(\theta_0) = -2\alpha^2 \frac{\cos^2\theta_0}{\sin^4\theta_0}
\end{equation}
Hence
$$f''_\alpha(\theta_0) -i \nu_\alpha''(\theta_0) = 2\alpha^2 \frac{\cos\theta_0}{\sin^4\theta_0}(\cos\theta_0 -i\sin \theta_0) = 2 (1-\alpha^2)^{1/2}\alpha^{-2} e^{i(\frac{\pi}{2}-\arcsin\alpha)}$$
We now invoke stationary phase/Laplace method to obtain
\begin{equation}
\label{E:SP4}
g_2(\alpha, h^{-1}) \approx -(2\alpha^2)^{\frac32-\sigma} e^{h^{-1}(-f_\alpha(\theta_0)+i\nu_\alpha(\theta_0))} e^{i(\sigma+\frac12)\theta_0} (\sin \theta_0)^{\sigma-\frac32} I(\alpha, h^{-1})
\end{equation}
where
\begin{equation}
\label{E:SP5}
\begin{aligned}
I(\alpha,h^{-1}) &= \int e^{-\frac12h^{-1}(f_\alpha''(\theta_0)-i\nu_\alpha''(\theta_0))(\theta-\theta_0)^2} \, d\theta \\
&= (2\pi)^{1/2} [ h^{-1}(f_\alpha''(\theta_0)-i\nu_\alpha(\theta_0))]^{-1/2} 
\end{aligned}
\end{equation}
Substituting $\theta_0=\arcsin\theta_0$ and \eqref{E:SP1}, \eqref{E:SP2}, \eqref{E:SP3} into \eqref{E:SP4}, \eqref{E:SP5}, we obtain
\begin{align*}
g_2(\alpha, h^{-1}) & \approx -2^{3/2} \pi^{1/2} e^{i\pi/4}  h^{1/2} e^{ih^{-1}(-\frac12-\alpha^2+\log(2\alpha))}e^{i\sigma \arcsin\alpha}  \\
& \qquad \cdots \times (1-\alpha^2)^{-1/4} \alpha^{\frac12-\sigma} e^{-h^{-1}[\alpha(1-\alpha^2)^{1/2}+\arcsin\alpha]}
\end{align*}
which is the dominant contribution for $\alpha<1$.  

Next, consider $\gamma_3$.  We parameterize it as $z=is$, where $s$ goes from $s=0$ to $s=2\alpha^2$.  Then
$$p_\alpha(is) = i\varphi_\alpha(s) - \tfrac12\pi$$
where
$$\varphi_\alpha(s) \defeq \tfrac18 \alpha^{-2}s^2 - s+ \log s$$
Then
$$g_3(\alpha, h^{-1}) = e^{-\pi/(2h)} e^{-i\pi\sigma/2}e^{\pi i/4}  \int_0^{2\alpha^2} e^{ih^{-1}\varphi_\alpha(s)} s^{-\sigma-\frac12} \, ds$$
Since $\varphi_\alpha'(s) = \frac14 \alpha^{-2} s - 1 + s^{-1}$, we find that when $\alpha \geq 1$, there is a stationary point in the interval $0< s \leq 2\alpha^2$ at
$$s_0 = 2\alpha^2( 1- \sqrt{1-\alpha^{-2}}) = \frac{2}{1+\sqrt{1-\alpha^{-2}}}$$
and
$$\varphi''(s_0) = -\tfrac12\sqrt{1-\alpha^{-2}}(1+ \sqrt{1-\alpha^{-2}})$$
For $\alpha \gg 1$, this takes the asymptotic form 
$$s_0 = 1+ O(\alpha^{-2}) \,, \qquad \varphi(s_0) = -1 + O(\alpha^{-2})\,, \qquad \varphi''(s_0) = -1 + O(\alpha^{-2})$$    
By stationary phase
\begin{align*}
\int_0^{2\alpha^2} e^{ih^{-1}\varphi_\alpha(s)} s^{-\sigma-\frac12} \, ds &\approx e^{ih^{-1}\varphi_\alpha(s_0)} s_0^{-\sigma-\frac12} \int e^{\frac12 i h^{-1}\varphi_\alpha''(s_0)(s-s_0)^2} \, ds \\
&= e^{ih^{-1}\varphi_\alpha(s_0)} s_0^{-\sigma-\frac12} h^{-1/2} (-\varphi_\alpha''(s_0))^{1/2}  (2\pi)^{1/2} e^{-i\pi/4}
\end{align*}
and hence
$$g_3(\alpha, h^{-1}) \approx  (2\pi)^{1/2} e^{-\pi/(2h)} e^{ih^{-1}\varphi_\alpha(s_0)} s_0^{-\sigma-\frac12} h^{-1/2} (-\varphi_\alpha''(s_0))^{1/2}  $$

In the case where $\alpha \gg h^{-1/2}$, we can use the asymptotic forms to simplify
$$g_3(\alpha, h^{-1}) \approx (2\pi)^{1/2}  e^{-\pi/(2h)} e^{-ih^{-1}} h^{-1/2}  $$
\end{proof}

\def\cprime{$'$} \def\cprime{$'$}
\providecommand{\bysame}{\leavevmode\hbox to3em{\hrulefill}\thinspace}
\providecommand{\MR}{\relax\ifhmode\unskip\space\fi MR }
\providecommand{\MRhref}[2]{%
  \href{http://www.ams.org/mathscinet-getitem?mr=#1}{#2}
}
\providecommand{\href}[2]{#2}


\begin{thebibliography}{HTMG94}


\bibitem[ADFT04]{MR2037249}
\bysame, \emph{Blow-up solutions for the {S}chr\"odinger equation in dimension
  three with a concentrated nonlinearity}, Ann. Inst. H. Poincar\'e Anal. Non
  Lin\'eaire \textbf{21} (2004), no.~1, 121--137. \MR{2037249 (2004k:35305)}

\bibitem[BCR99]{MR1699715} Chris J. Budd, Shaohua Chen, and Robert D. Russell, \emph{New self-similar solutions of the nonlinear Schr\"odinger equation with moving mesh computations}, J. Comput. Phys. 152 (1999), no. 2, 756--789. 

\bibitem[Fib15]{MR3308230}
Gadi Fibich, \emph{The nonlinear {S}chr\"odinger equation}, Applied
  Mathematical Sciences, vol. 192, Springer, Cham, 2015, Singular solutions and
  optical collapse. \MR{3308230}

\bibitem[Fra85]{MR807329}
G.~M. Fra{\u\i}man, \emph{Asymptotic stability of manifold of self-similar
  solutions in self-focusing}, Zh. \`Eksper. Teoret. Fiz. \textbf{88} (1985),
  no.~2, 390--400. \MR{807329 (86m:78002)}
  
  
\bibitem[KL95]{MR1349311}
Nancy Kopell and Michael Landman, \emph{Spatial structure of the focusing
  singularity of the nonlinear {S}chr\"odinger equation: a geometrical
  analysis}, SIAM J. Appl. Math. \textbf{55} (1995), no.~5, 1297--1323.
  \MR{1349311 (96g:35176)}


\bibitem[LPSS88]{MR966356}
M.~J. Landman, G.~C. Papanicolaou, C.~Sulem, and P.-L. Sulem, \emph{Rate of
  blowup for solutions of the nonlinear {S}chr\"odinger equation at critical
  dimension}, Phys. Rev. A (3) \textbf{38} (1988), no.~8, 3837--3843.
  \MR{966356 (89k:35218)}


\bibitem[MR03]{MR1995801}
F.~Merle and P.~Raphael, \emph{Sharp upper bound on the blow-up rate for the
  critical nonlinear {S}chr\"odinger equation}, Geom. Funct. Anal. \textbf{13}
  (2003), no.~3, 591--642. \MR{1995801 (2005j:35207)}

\bibitem[MR04]{MR2061329}
Frank Merle and Pierre Raphael, \emph{On universality of blow-up profile for
  {$L^2$} critical nonlinear {S}chr\"odinger equation}, Invent. Math.
  \textbf{156} (2004), no.~3, 565--672. \MR{2061329 (2006a:35283)}

\bibitem[MR05a]{MR2150386}
\bysame, \emph{The blow-up dynamic and upper bound on the blow-up rate for
  critical nonlinear {S}chr\"odinger equation}, Ann. of Math. (2) \textbf{161}
  (2005), no.~1, 157--222. \MR{2150386 (2006k:35277)}

\bibitem[MR05b]{MR2116733}
\bysame, \emph{Profiles and quantization of the blow up mass for critical
  nonlinear {S}chr\"odinger equation}, Comm. Math. Phys. \textbf{253} (2005),
  no.~3, 675--704. \MR{2116733 (2006m:35346)}

\bibitem[MR06]{MR2169042}
\bysame, \emph{On a sharp lower bound on the blow-up rate for the {$L^2$}
  critical nonlinear {S}chr\"odinger equation}, J. Amer. Math. Soc. \textbf{19}
  (2006), no.~1, 37--90 (electronic). \MR{2169042 (2006j:35223)}

\bibitem[MRS10]{MR2729284}
Frank Merle, Pierre Rapha{\"e}l, and Jeremie Szeftel, \emph{Stable self-similar
  blow-up dynamics for slightly {$L^2$} super-critical {NLS} equations}, Geom.
  Funct. Anal. \textbf{20} (2010), no.~4, 1028--1071. \MR{2729284
  (2011m:35294)}

\bibitem[Per01]{perelman2001formation}
Galina Perelman, \emph{On the formation of singularities in solutions of the
  critical nonlinear Schr{\"o}dinger equation}, Annales Henri Poincar{\'e},
  vol.~2, Springer, 2001, pp.~605--673.

\bibitem[Rap05]{MR2122541}
Pierre Raphael, \emph{Stability of the log-log bound for blow up solutions to the critical non linear {S}chr\"odinger equation}, Math. Ann. \textbf{331} (2005), no.~3, 577--609. \MR{2122541 (2006b:35303)}

\bibitem[RK03]{MR1975789}  Vivi Rottsch\"afer and Tasso J. Kaper, \emph{Geometric theory for multi-bump, self-similar, blowup solutions of the cubic nonlinear Schr\"odinger equation}, Nonlinearity 16 (2003), no. 3, 929--961. 
{MR1398655}
\bibitem[Sla96]{MR1398655} S. Yu Slavyanov, \emph{Asymptotic solutions of the one-dimensional Schr\"odinger equation}. Translated from the 1990 Russian original by Vadim Khidekel. Translations of Mathematical Monographs, 151. American Mathematical Society, Providence, RI, 1996. xvi+190 pp. ISBN: 0-8218-0563-3 .

\bibitem[SS99]{MR1696311}
Catherine Sulem and Pierre-Louis Sulem, \emph{The nonlinear {S}chr\"odinger
  equation}, Applied Mathematical Sciences, vol. 139, Springer-Verlag, New
  York, 1999, Self-focusing and wave collapse. \MR{1696311 (2000f:35139)}

\end{thebibliography}
\end{document}